\documentclass[12pt]{amsart}

\usepackage{amsmath, amssymb, amscd, verbatim, xspace,amsthm}
\usepackage{latexsym, epsfig, color}
\usepackage[hidelinks]{hyperref}

\usepackage[OT2,T1]{fontenc}
\DeclareSymbolFont{cyrletters}{OT2}{wncyr}{m}{n}
\DeclareMathSymbol{\Sha}{\mathalpha}{cyrletters}{"58}

\newcommand{\ba}{\begin{align*}}
\newcommand{\ea}{\end{align*}}

\newcommand{\A}{\ensuremath{{\mathbb{A}}}}
\newcommand{\C}{\ensuremath{{\mathbb{C}}}}
\newcommand{\Z}{\ensuremath{{\mathbb{Z}}}\xspace}
\renewcommand{\P}{\ensuremath{{\mathbb{P}}}}

\newcommand{\Q}{\ensuremath{{\mathbb{Q}}}}
\newcommand{\R}{\ensuremath{{\mathbb{R}}}}
\newcommand{\F}{\ensuremath{{\mathbb{F}}}}

\newcommand{\E}{\ensuremath{{\mathbb{E}}}}

\newcommand{\ra}{\rightarrow}

\newcommand\Conf{\operatorname{Conf}}
\newcommand\Hom{\operatorname{Hom}}

\newcommand\Aut{\operatorname{Aut}}

\newcommand\im{\operatorname{im}}

\newcommand\Gal{\operatorname{Gal}}
\newcommand\Nm{\operatorname{Nm}}

\newcommand\Sur{\operatorname{Sur}}

\newcommand\Tr{\operatorname{Tr}}

\newcommand\tensor{\otimes}
\newcommand\isom{\simeq}
\newcommand\sub{\subset}

\newcommand\Disc{\operatorname{Disc}}
\newcommand\GL{\operatorname{GL}}

\newcommand\PSL{\operatorname{PSL}}

\newcommand\SL{\operatorname{SL}}

\newcommand\Spec{\operatorname{Spec}}

\newcommand\Frob{\operatorname{Frob}}

\renewcommand\O{\mathcal{O}}

\newcommand\bq{\begin{equation}}
\newcommand\eq{\end{equation}}

\newtheorem{proposition}{Proposition}[section]
\newtheorem{theorem}[proposition]{Theorem}

\newtheorem{question}[proposition]{Question}
\newtheorem{lemma}[proposition]{Lemma}

\newtheorem{conjecture}[proposition]{Conjecture}

\theoremstyle{remark}
\newtheorem{remark}[proposition]{Remark}

\usepackage{tikz}

\usepackage{fullpage,pdfsync}

\usepackage{url}

\newenvironment{notation}{\vspace{2 ex}{\noindent{\bf Notation. }}}{\vspace{2 ex}}

\newtheorem{definition}[proposition]{Definition}

\newcommand\twist{twist\xspace}
\newcommand\sa{admissible\xspace}

\newcommand\gd{good\xspace}
\newcommand{\un}{\operatorname{un}}
\renewcommand{\v}{\infty}

\newcommand{\CCHur}{\mathsf{CHur}}

\newcommand{\SC}{C}

\title{Nonabelian Cohen-Lenstra Moments}

\author{Melanie Matchett Wood}
\address{Department of Mathematics\\
University of Wisconsin-Madison \\ 480 Lincoln Drive \\
Madison, WI 53705 USA}  
\email{mmwood@math.wisc.edu}

\begin{document}

\begin{abstract}
In this paper we give a conjecture for the average number of unramified $G$-extensions of a quadratic field for any finite group $G$.  The Cohen-Lenstra heuristics are the specialization of our conjecture to 
 the case that $G$ is abelian of odd order.  We prove a theorem towards the function field analog of our conjecture, and give additional motivations for the conjecture including the construction of a lifting invariant for the unramified $G$-extensions that takes the same number of values as the predicted average and an argument using the Malle-Bhargava principle. We note that for even $|G|$, corrections for the roots of unity in $\Q$ are required, which can not be seen when $G$ is abelian. 
\end{abstract}
\maketitle

\section{Introduction}

In 1984, Cohen and Lenstra \cite{Cohen1984} gave heuristics for the distribution of class groups of quadratic fields.  By class field theory, the class group is the Galois group of the maximal unramified abelian extension, and much of the work on the Cohen-Lenstra heuristics has been through the lens of unramified abelian extensions.  It is then natural to ask about the nonabelian version, that is, the Galois group of the maximal unramified extension.  Given a quadratic field $K$, a Galois $G$ extension $L$ of $K$ has Galois closure $\tilde{L}$ over $\Q$ with Galois group some subgroup $G'$ of the wreath product $G\wr S_2$.
The following question, asked by Bhargava
at the 2011 AIM Workshop on the Cohen-Lenstra heuristics and in 
 \cite[Section 1.2]{Bhargava2014b}, is a rather comprehensive question about the distribution of these Galois groups.  Given a finite group $G$ and a subgroup $G'\sub G\wr S_2$, what is the average $E^-(G,G')$ (resp. $E^+(G,G')$) of the number of unramified Galois $G$-extensions $L/K$ such that $\Gal(\tilde{L}/\Q)$ is $G'$, per imaginary (resp. real) quadratic field $K$ (ordered by discriminant)?  In this paper, we give a conjecture for the answer, prove a theorem towards the function field analog, and give other motivations for our conjecture.

We say that $G'$ is an \emph{\sa} subgroup of $G\wr S_2$ if $G'$ contains $(1,\sigma)$ (where $\sigma$ generates $S_2$), $G'$ is generated by order $2$ elements with non-trivial image in $S_2$, and the kernel of $\pi: G' \ra S_2$ surjects onto the first factor of $G$.  We call $G'$ \emph{\gd} if it has a unique conjugacy class $c$ of order two elements not in $\ker \pi$.  See Section~\ref{S:types} for some examples of \sa and good $G'$.  The group $\Aut(G)$ acts naturally on $G\wr S_2$ and we write $\Aut_{G'}(G)$ for the subgroup that fixes the subgroup $G'$ of $G\wr S_2$  setwise.

\begin{conjecture}\label{C:E}
Given a finite group $G$, and an \sa subgroup $G'\sub G\wr S_2$, if $G'$ is good then
\begin{align*}
E^{-}(G,G')=
\frac{|H_2(G',c)[2]|}{|\Aut_{G'} (G)|} \quad \textrm{and} \quad 
E^{+}(G,G')=
\frac{|H_2(G',c)[2]|}{|c||\Aut_{G'} (G)|}
\end{align*}
and if $G'$ is not good then
$$
E^{\pm}(G,G')=\infty.
$$
\end{conjecture}

See Section~\ref{S:Hdefs} for the definition of the \emph{reduced Schur multiplier} $H_2(G',c)$, which is a quotient of the group homology (Schur multiplier)
$H_2(G',\Z)$.  The ``$[2]$'' denotes the $2$-torsion.  See Section~\ref{S:Hchart} for a chart of some values of $H_2(G',c)$, and note $|H_2(G',c)[2]|$ is often $1$.  In this paper, we will see that the factor $|H_2(G',c)[2]|$ appears as a ``correction'' for the roots of unity in $\Q$.
See Section~\ref{S:Edefs} for a precise definition of $E^{\pm}(G,G')$, and in particular how we make $\Gal(\tilde{L}/\Q)$ a subgroup of $G\wr S_2$.

If $G'$ is not \sa, then there are no unramified $G$-extensions $L/K$ with $\Gal(\tilde{L}/\Q)$ giving $G'$  for any $K$ (Proposition~\ref{P:whysa}), and so $E^{\pm}(G,G')=0$.

  If $|G|$ is odd, then $G'$ is always good (by the Sylow theorems or by Schur-Zassenhaus) and $|H_2(G',c)[2]|=1$ (see \cite[Example 9.3.2]{Ellenberg2012}).  Thus, in the imaginary quadratic case when $|G|$ is odd, this conjecture reduces to one of Boston and the author \cite[Conjecture 1.6]{Boston2017}. (To compare, note that $\tilde{E}^-(G,G')$ below in Equation~\eqref{E:relateE} is the same as the count of surjections in \cite[Conjecture 1.6]{Boston2017} to $\ker(G' \ra S_2).$)
    If $G$ is abelian, then there is only one \sa $G'=G\rtimes_{-1} S_2$, and $\Aut_{G'}(G)=\Aut(G)$, and $|c|=|G|$, and so  $|\Aut(G)| E^{\pm}(G,G')$ are the usual $G$-moments of the Cohen-Lenstra heuristics, that is the average number of surjections from the class group of $K$ to $G$.
   When $G$ is abelian and of odd order, the prediction of Conjecture~\ref{C:E} agrees with that of the Cohen-Lenstra heuristics  
  (see Section~\ref{S:prev} for further discussion on previous work on these moments).  

In this paper, we give three main motivations for Conjecture~\ref{C:E}.  Our first motivation is the following theorem in the direction of the function field analog of Conjecture~\ref{C:E}.   We define $E^{\pm}_{\F_q(t),q^{2n}}(G,G')$ as the analog of $E^{\pm}(G,G')$ in which $\Q$ is replaced by $\F_q(t)$ (for a finite field $\F_q$) and we only consider $K$ with $\Nm \Disc K=q^{2n}$ (see Section~\ref{S:Edefs}).

\begin{theorem}\label{T:FF}
Given a finite group $G$ and an \sa subgroup $G'\sub G\wr S_2$ (where $G'$ has trivial center),
and $c$ the set of order $2$ elements of $G' \setminus \ker(G'\ra S_2)$,
for $n$ sufficiently large we have
\begin{align*}
\lim_{\substack{q\ra\infty\\(q,|G'|)=1}}
\frac{E^{\pm}_{\F_q(t),q^{2n}}(G,G')}{|H_2(G',c)[|\mu_{\F_q(t)}|]|}
\begin{cases}
=\frac{1}{|\Aut_{G'} (G)|} &\textrm{if $G'$ good, $-$ case,}\\
=\frac{1}{|c||\Aut_{G'} (G)|} &\textrm{if $G'$ good, $+$ case, and}\\
\gg n^{N_{G'}-1} &\textrm{otherwise,}
\end{cases}
\end{align*}
where $\mu_{\F_q(t)}$ is the group of roots of unity of $\F_q(t)$, so $|\mu_{\F_q(t)}|=q-1$, and $N_{G'}$ is the number of conjugacy classes of order $2$ elements of $G'$ not in $\ker(G'\ra S_2)$.
\end{theorem}

We prove this theorem in Section~\ref{S:FF}.  The method of proof is by counting $\F_q$ points on moduli spaces for the relevant extensions, an idea going back to unpublished work of J.-K. Yu and built upon by Achter \cite{Achter2006}.
The recent breakthroughs of Ellenberg, Venkatesh, and Westerland \cite{Ellenberg2016, Ellenberg2012} have made this method much more powerful.  In the proof of Theorem~\ref{T:FF}, we use the construction of the relevant moduli schemes from \cite{Ellenberg2012}, which builds on the work of Romagny and Wewers \cite{Romagny2006}, and the description of the components of those schemes from \cite{Ellenberg2012}.   The condition in Theorem~\ref{T:FF} that $G'$ has a trivial center comes from a technical condition in \cite[Section 8.6.2]{Ellenberg2012} required to ensure that the Hurwitz space parametrizing $G'$ covers of the line is an actual scheme and not just a stack.

The moduli spaces used in the proof of Theorem~\ref{T:FF} parametrize \emph{rigidified} extensions that come with a chosen isomorphism of  $\Gal(\tilde{L}/\F_q(t))$ with $G'$.  
We let $\tilde{E}^{\pm}(G,G')$ count these rigidified extensions (see Section~\ref{S:Edefs} for a precise definition), and when $G'$ is good
\begin{equation}\label{E:relateE}
\tilde{E}^{-}(G,G')= |\Aut_{G'} (G)|E(G,G') \quad \textrm{and} \quad \tilde{E}^{+}(G,G')= |c||\Aut_{G'} (G)|E(G,G') .
\end{equation}
When $q\ra \infty$ for fixed $n$, if we have a bound on the cohomology of the moduli spaces
that is independent of $q$, then 
 $\tilde{E}^{\pm}_{\F_q(t),q^{2n}}(G,G')$  will go to the number of components of the moduli space.  By default, we might guess there is $1$ component, but in this case there are $|H_2(G',c)[|\mu_{\F_q(t)}|]|$ components.  This makes it apparent that the $|H_2(G',c)[2]|$ appearing in Conjecture~\ref{C:E} is a ``correction factor''  for roots of unity in the base field (i.e. the $2$ roots of unity in $\Q$).   
Even in the usual Cohen-Lenstra-Martinet heuristics for quadratic extensions, which are about the case of abelian $G$ of odd order above,
it is known (see \cite{Malle2008, Garton2015, Adam2015}) that corrections are needed for roots of unity appearing in the base field.  When $G$ is abelian and of even order, the only \sa $G'$ cannot be good (see Lemma~\ref{L:goodab}). So, these corrections for the roots of unity in $\Q$ never appeared for abelian $G$.  However, the theorem above in the non-abelian case now makes clear that even the roots of unity in $\Q$ need to be corrected for in these heuristics.

The limits in Theorem~\ref{T:FF} are not the same as the limits in the direct function field analog of Conjecture~\ref{C:E}.  In particular, in the analog of Conjecture~\ref{C:E} there should be a fixed $q$ and a limit in $n$.  In Theorem~\ref{T:FF}, we have a fixed $n$ and a limit in $q$.  This is not a minor  difference.  However, we still find Theorem~\ref{T:FF} to be a strong suggestion for Conjecture~\ref{C:E} (and its function field analog).  When $G$ is abelian and of odd order, then Theorem~\ref{T:FF} is roughly the same as a result of Achter \cite{Achter2006}.
Also, when $|G|$ is odd, Ellenberg, Venkatesh, and Westerland \cite{Ellenberg2016} (for abelian $G$, in the imaginary quadratic case), the author \cite{Wood2017b} (for abelian $G$, in the real quadratic case)
and Boston and the author \cite{Boston2017} (for any $G$, in the imaginary quadratic case) prove a result
like Theorem~\ref{T:FF}, but with a limit in $n$, before a limit in $q$, making it closer to the analog of 
Conjecture~\ref{C:E}.

The second main motivation for Conjecture~\ref{C:E} is based on the components that give the answer in the above theorem.  The components in the moduli space divide the unramified rigidified $G$ extensions of type $G'$ of a quadratic extension $K/\F_q(t)$ into $|H_2(G',c)[|\mu_{\F_q(t)}|]|$ natural classes.  Over $\Q$, there is no known analog of the moduli space to break into components.  However, as a motivation for Conjecture~\ref{C:E}, in Section~\ref{S:lift} we do show that for any global field $F$, there is an invariant that divides 
the unramified rigidified $G$ extensions of type $G'$ of a quadratic extension $K/F$ 
 into $|H_2(G',c)[|\mu_{F}|]|$ classes, via a construction that we then prove agrees over $\F_q(t)$ with the division into components. 
Moreover, in Section~\ref{S:Refine}, we refine Conjecture~\ref{C:E} to a conjecture about the number of extensions with each possible lifting invariant (which is supported by Theorem~\ref{T:SFF}, a refined version of Theorem~\ref{T:FF}, that we prove in Section~\ref{S:FF}).

Venkatesh and Ellenberg \cite[Section 2.5]{Venkatesh2010} previously defined a lifting invariant of extensions, and ours is motivated  by theirs, but we cannot apply their invariant here because it requires the order of the elements of $c$ to be relatively prime to $|H_2(G',c)|$.  Further, their invariant lies in $H_2(G',c)$, but
we prove that ours actually lies in $H_2(G',c)[|\mu_{F}|],$ motivating the precise value appearing in Conjecture~\ref{C:E}.  Venkatesh and Ellenberg originally had the idea that these Schur multipliers should provide some kind of correction to asymptotic constants \cite[Section 2.4]{Venkatesh2010}.

 The final motivation is given in Section~\ref{S:MB} and is based on  the Malle-Bhargava principle for counting number fields.  We explain how this principle for producing conjectures about the asymptotic count of number fields agrees with the division in Conjecture~\ref{C:E} into the finite and infinite cases.  We discuss at some length the applicability of these principles here, since they have known counterexamples.  In fact, these principles also suggest when $G'$ is not good that the growth of $E^{\pm}(G,G')$ with $X$ is as seen in the lower bound of Theorem~\ref{T:FF}.
 
As further evidence for Conjecture~\ref{C:E}, in the appendix with P.~M.~Wood we present some empirical evidence for the correction for roots of unity, i.e. the factor $|H_2(G',c)[2]|$ in Conjecture~\ref{C:E}.  The smallest $G$ for which this factor is non-trivial is $G=A_4$, and in this case (with the appropriate $G'$), we estimate $E^-(A_4,G')$ by considering fields up to absolute discriminant $2^{32}$ and show evidence that $E^-(A_4,G')>|\Aut_{G'}(A_4)|^{-1}.$

\subsection{Cases of Conjecture~\ref{C:E} known}\label{S:prev}
Since Conjecture~\ref{C:E} specializes to the Cohen-Lenstra moments when $G$ is abelian, it is of course known in very few cases.
When $G=\Z/3\Z$ it is known by the theorem of Davenport and Heilbronn \cite{Davenport1971}, and in this case for $\Q$ replaced by a global field by Datskovsky and Wright 
\cite{Datskovsky1988}.
When $(G,G')$ is $(A_4,S_4),(A_5,S_5),$ or $(S_n,S_n\times C_2)$ for $n=3,4,5$,  Conjecture~\ref{C:E} is known by work of Bhargava \cite[Theorem 1.4]{Bhargava2014b}.  In the first two of these cases $G'$ is good, and in the last three $G'$ is not good.
When $G$ is the quaternion group of order $8$, there is one \sa $G'$ of order $16$ (SmallGroup(16,13) in the Small Groups Library \cite{SmallGroups}), which is not good, and when $G$ is the dihedral group of order $8$, the only \sa $G'$ of order $16$ is $D_8\times C_2$, which is also not good.
Conjecture~\ref{C:E} is known for these two $(G,G')$ by Alberts \cite[Corollary 4.3]{Alberts2016}.  In Section~\ref{S:Ex}, we prove that ${E}^\pm(C_2^k, C_2^{k+1})=\infty$, as predicted by 
Conjecture~\ref{C:E}.

\subsection{Moments determining the distribution}
If we combine appropriately the extensions counted by ${E}^\pm(G,G')$ over all \sa $G'$, we would call the result the $G$-moment of the Galois group $\Gal(K^{un}/K)$ of the maximal unramified extension $K^{un}$.  When $G$ is abelian, these quantities have a close connection to moments (in the usual sense) of the group invariants (see \cite[Section 3.3]{Clancy2015}).  We use the same language when $G$ is non-abelian by analogy.  

If $\mathcal{G}$ is a set of isomorphism classes of profinite groups, and $X$ is a random group from $\mathcal{G}$,
then for several  $\mathcal{G}$, the group-indexed moments (that arise in Cohen-Lenstra type conjectures)
$\E(\#\Sur(X,G))$ for all $G\in \mathcal{G}$  determine a unique probability distribution for $X$, even though they grow too fast when translated to usual moments to use standard probabilistic methods (see \cite[Section 1.5]{Wood2017}).  Such cases include $\mathcal{G}$ elementary abelian $p$-groups due to Fouvry and Kl\"{u}ners \cite[Theorems 1 and 2]{Fouvry2006},
abelian $p$-groups due to Ellenberg, Venkatesh, and Westerland \cite[Lemma 8.1]{Ellenberg2016},
abelian groups whose order is only divisible by a given finite set of primes by the author \cite[Theorem 8.3]{Wood2017},
and  pro-$p$-groups due to Boston and the author \cite[Theorem 1.4]{Boston2017}.

In the generality of Conjecture~\ref{C:E}, it is an interesting open question whether the moments determine a distribution in an appropriate sense.  It is even an interesting open question to construct a measure on profinite groups that gives the moments appearing in Conjecture~\ref{C:E}.

\subsection{Outline of the paper}
In Section~\ref{S:Edefs}, we give a careful definition of  $E^\pm(G,G')$ and  $\tilde{E}^\pm(G,G')$, and go over other foundational definitions.    In Section~\ref{S:lift}, we give the definition of the lifting invariant that divides unramified $G$ extensions of type $G'$ into $H_2(G',c)[|\mu_{F}|]$ classes.  In Section~\ref{S:FF}, we prove a refinement of Theorem~\ref{T:FF}.  In Section~\ref{S:Refine}, we refine Conjecture~\ref{C:E} by lifting invariant.
In Section~\ref{S:MB}, we explain how the Malle-Bhargava principle suggests the division in Conjecture~\ref{C:E} into finite and infinite cases. In Section~\ref{S:Ex}, we prove that ${E}^\pm(C_2^k, C_2^{k+1})=\infty$, to give another example where
Conjecture~\ref{C:E} holds.  In Section~\ref{S:chart}, we give charts of the values of invariants important in this paper (number of conjugacy classes in $c$, structure of $H_2(G',c)$, and size of the center of $G'$)  for small groups.

\subsection{Notation}
If $A$ is a group and $B$ is a permutation group on $d$ elements, then $A\wr B$ denotes the wreath product, that is $A^d \rtimes B$, where $B$ acts on $A^d$ by permuting the factors through its permutation action.  We write $S_n$ for the symmetric group on $n$ elements.  We use $\Sur(G,H)$ to denote the surjective homomorphisms from $G$ to $H$, which are required to be continuous when either is a profinite group. For a field $K$, we write $\bar{K}$ for its separable closure and $G_K$ for $\Gal(\bar{K}/K)$.  For a global field $K$ and place $v$, we write $K_v$ for the completion of $K$ at $v$.

\section{Definitions and basic background}\label{S:basic}
\subsection{Definition of $E^{\pm}(G,G')$ and $\tilde{E}^{\pm}(G,G')$}\label{S:Edefs}
In this section we will precisely define the $E^{\pm}(G,G')$ used in Conjecture~\ref{C:E}.
Let $Q$ be $\Q$ or $\F_q(t)$ and $\v$ the  place at infinity. 
We fix a separable closure $\bar{Q}$ of $Q$, and an inertia subgroup for $\infty$ in $\Gal(\bar{Q}/Q)$.
When $Q=\F_q(t)$ we call a quadratic extension imaginary if it is ramified at infinity and the completion at infinity is isomorphic (as an extension of $\F_q((t^{-1}))$) to the completion of $\F_q(\sqrt{t})$ at infinity, and we call a quadratic extension real if it is split completely at infinity.  (Note that in the function field case, ``real'' and ``imaginary'' as defined here do not exhaust the possible completions at infinity.)
  If $K\sub \bar{Q}$ is a quadratic extension, we define $K^{\un,\v}\sub \bar{Q}$ to be the maximal unramified extension of $K$, split completely over $\v$.  (Of course if $\v$ is an archimedian place, unramified is the same as split completely.)
Let $G_K^{\un,\v}:=\Gal(K^{\un,\v}/K)$.   
  Given  $L\sub K^{\un,\v}$, with $L$ an extension of $K$, we have $\tilde{L}\sub \bar{Q}$, the Galois closure of $L$ over $Q$.  
  
  \subsubsection{Imaginary quadratic}\label{S:imagsetup} If $K/Q$ is imaginary quadratic, note the inertia subgroup over $\v$ in 
  $\Gal(\tilde{L}/Q)$ is non-trivial, but has trivial intersection with the index two subgroup $\Gal(\tilde{L}/K)$.  Thus
  the inertia subgroup is order $2$ and we let $y\in \Gal(\tilde{L}/Q)$ be its non-trivial element.  
  We next will define a map
\begin{equation}\label{E:getphi}
 \phi:  \Gal(\tilde{L} /Q) \ra (\Gal(L/K) \times \Gal(L/K) ) \rtimes_\sigma S_2=: \Gal(L/K) \wr S_2,
\end{equation}
 where we write $\sigma$ for the generator of the cyclic group $S_2$ of order $2$ as well as the automorphism of 
 $\Gal(L/K) \times \Gal(L/K)$ that switches the factors.
 We first define   
 $$
 \phi|_{\Gal(\tilde{L} /K)}: \Gal(\tilde{L} /K) \ra \Gal(L/K) \times \Gal(L/K) ,
$$  
  where the first projection is the usual projection, and the second is conjugation by $y$ in $\Gal(\tilde{L} /Q) $ followed by the usual projection.  
We note that $ \phi|_{\Gal(\tilde{L} /K)}$ is injective and has image that surjects onto the first factor of $\Gal(L/K)$.
Then, we extend $ \phi|_{\Gal(\tilde{L} /K)}$ to a map $\phi :\Gal(\tilde{L} /Q)\ra \Gal(L/K) \wr S_2$, by sending $y$ to $(1,\sigma)$, and we can check this 
gives a homomorphism.
 We note that $\phi$ is injective on  $\Gal(\tilde{L} /Q)$, and its image surjects on $S_2$.

 We see above that a surjection in $\Sur(G_K^{\un,\v},G),$ i.e. an extension $L\sub K^{\un,\v}$ of $K$ and an isomorphism $ \Gal(L/K) \stackrel{\sim}{\ra} G$, gives an injection  $\Gal(\tilde{L}/Q) \ra (G\times G) \rtimes_\sigma S_2=:G \wr S_2$.  For a subgroup $G'\sub G \wr S_2$, we say a surjection $G_K^{\un,\v}\ra G$ is \emph{type} $G'$ if $G'$ is the image of the associated $\Gal(\tilde{L}/Q) \ra G \wr S_2$.  The following diagram shows the relationships between the fields and groups involved for the reader's convenience.

\begin{tikzpicture}[scale=2]
 \node (Q)   at (1,0)     {$Q$};
  \node (K)  at (0,1)         {$K$};
   \node (L)   at (0,2)       {$L$};
 \node (tL)   at (3,3)               {$\tilde{L}$};
\draw[-] (Q) to node[below] {$S_2$} (K);
 \draw (Q)     to node[right] {$G'$} (tL);
 \draw (K)    to node[left] {$G$} (L);
 \draw (K)    to node[right] {$\ker \pi$} (tL);
 \draw (L)  -- (tL);
\end{tikzpicture} 

    Given $Q$, we let $IQ_X$ be the set of  imaginary quadratic $K\sub \bar{Q}$ of $\Nm \Disc (K/Q)\leq X$.  We define, if the limit exists,
        $$
    \tilde{E}^-_{Q}(G,G')=\lim_{X\ra\infty} \frac{\sum_{K\in IQ_X}  \#\{\rho\in \Sur(G_K^{\un,\v}, G) | \rho \textrm{ is type }G'\}    }{\sum_{K\in IQ_X} 1}.
    $$
Let  $IQ_{=X}$ be the set of  imaginary quadratic $K\sub \bar{Q}$ of $\Nm \Disc (K/Q)= X$.
We define
      $$
   \tilde{E}^-_{Q,X}(G,G')=\frac{\sum_{K\in IQ_{=X}}
   \#\{\rho\in \Sur(G_K^{\un,\v}, G) | \rho \textrm{ is type }G'\}    }{\sum_{K\in IQ_{=X}} 1}.
    $$
    (The quantity $\tilde{E}^-_{Q,X}(G,G')$ is mainly of interest in the function field case, and in that case
  it is elementary to show that if $\lim_{n\ra\infty} \tilde{E}^-_{Q,q^{2n}}(G,G')$  exists then    $\tilde{E}^-_{Q}(G,G')$ exists and has the same value.)
If we omit the $Q$ subscript we mean $Q=\Q$.

In this paper, it is cleaner to include the data of an isomorphism between $\Gal(L/K)$ and $G$,
but of course one could also say that gives an overcounting of the true object of interest $L$.  Note that $\Aut(G)$ acts naturally on $G\wr S_2$, and write $\Aut_{G'}(G)$ for the subgroup of $\Aut(G)$ that fixes (setwise) the subgroup $G'\sub G\wr S_2$.  Then, we see a field $L$ is counted $\Aut_{G'}(G)$ times in $\tilde{E}^-(G,G').$
If $G$ is index $2$ in $G'$, then $\sigma$ gives an automorphism of $G$ and the elements $\Aut_{G'}(G)$ are equivalently the automorphisms of $G$ that commute with $\sigma$ (as an element of $\Aut(G)$).  
So we have
\begin{equation*}
E^-(G,G')=\frac{1}{|\Aut_{G'}(G)|}\tilde{E}^-(G,G').
\end{equation*}

  \subsubsection{Real quadratic}
  When $K/Q$ is real quadratic, there does not appear to be a natural map $\phi$ as in Equation~\eqref{E:getphi} because there is no clear way to choose $y$.  However, note that order $2$ elements of
  $\Gal(\tilde{L}/Q)\setminus \Gal(\tilde{L}/K)$ are the only possible images for inertia subgroups at any prime (since $\tilde{L}/K$ is unramified).  Of the real quadratic $K/Q$, none are unramified, and so there are order  $2$ elements of  $\Gal(\tilde{L}/Q)\setminus \Gal(\tilde{L}/K)$.  As in the last section, any such choice $y$ of an order $2$ element of  $\Gal(\tilde{L}/Q)\setminus \Gal(\tilde{L}/K)$ gives an injection 
  $\phi: \Gal(\tilde{L}/Q)\ra \Gal(L/K)\wr S_2$.
  
  Then a surjection $\rho\in \Sur(G_K^{\un,\v},G),$ i.e. an extension $L\sub K^{\un,\v}$ of $K$ and an isomorphism $ \Gal(L/K) \stackrel{\sim}{\ra} G$, \emph{and} an order $2$ element $y$ of  $\Gal(\tilde{L}/Q)\setminus \Gal(\tilde{L}/K)$ (which we call a \emph{\twist} for $\rho$), gives an injection of $\Gal(\tilde{L}/Q) \ra (G\times G) \rtimes_\sigma S_2=:G \wr S_2$.  For a subgroup $G'\sub G \wr S_2$, we say $(\rho,y)$ is \emph{type} $G'$ if $G'$ is the image of the associated $\Gal(\tilde{L}/Q) \ra G \wr S_2$.
  
    Given $Q,\v$, we let $RQ_X$ (resp. $RQ_{=X}$) be the set of  real quadratic $K\sub \bar{Q}$ of $\Nm \Disc (K/Q)\leq X$ (resp. $\Nm \Disc (K/Q)= X$).  We define,  if the limit exists,
        $$
    \tilde{E}^+_{Q}(G,G')=\lim_{X\ra\infty} \frac{\sum_{K\in RQ_X}  \#\{(\rho,y) | \rho\in \Sur(G_K^{\un,\v}, G), y\textrm{ \twist for }\rho , (\rho,y) \textrm{ is type }G'\}    }{\sum_{K\in RQ_X} 1}.
    $$
and
      $$
   \tilde{E}^+_{Q,X}(G,G')=\frac{\sum_{K\in RQ_{=X}}  \#\{(\rho,y)  | \rho\in \Sur(G_K^{\un,\v}, G), y\textrm{ \twist for }\rho , (\rho,y) \textrm{ is type }G'\}      }{\sum_{K\in RQ_{=X}} 1}.
    $$
If we omit the $Q$ subscript we mean $Q=\Q$.

Again, given $G$ and $G'$, there is a group theoretic factor $A(G,G')$ that tells us how many different choices $(\rho,y)$ there are for a single a field $L$ that give type $G'$.  We have
$$
E^+(G,G')=\frac{1}{A(G,G')}\tilde{E}^+(G,G').
$$
We can write $A(G,G')$ more simply when $G'$ is good.
\begin{lemma}
If $G'$ is good for a finite group $G$, then
$$
A(G,G')=|c||\Aut_{G'}(G)|.
$$
\end{lemma}
\begin{proof}
Let $K/Q$ be a quadratic extension, and $L/K$ with $\rho: \Gal(L/K)\isom G$ with twist $y$.   Suppose that $(\rho,y)$ is type $G'$. 

Now let $z$ be an order $2$ element of $\Gal(\tilde{L}/Q)\setminus \Gal(\tilde{L}/K)$.  Since $G'$ is good, we have that $y$ and $z$ are conjugate in $\Gal(\tilde{L}/Q)$.  Further, we can assume $y=z^k$ for some $k \in \Gal(\tilde{L}/K)$ since
$\Gal(\tilde{L}/Q) = \Gal(\tilde{L}/K) \cup y\Gal(\tilde{L}/K)$ (where $z^k$ denotes right conjugation of $z$  by $k$). 
If we define $\psi :\Gal(L/K)\ra G$ so that $\psi(w)=\rho(k)^{-1}\rho(w)\rho(k)$, then in the map $\Gal(\tilde{L}/Q)\ra G'$ given by $(\psi,z)$, we have that for $x\in \Gal(\tilde{L}/K)$,
\begin{align*}
x \mapsto &(\psi(x),\psi(x^z))=(\psi(x),\psi(z^{-1}xz))=(\psi(x),\psi(ky^{-1} k^{-1}x kyk^{-1} ))\\
=&(\psi(k) \psi(k^{-1}xk) \psi(k)^{-1},\psi(k) \psi((k^{-1}xk)^y) \psi(k^{-1} ))\\
=&(\rho(k^{-1}xk) , \rho((k^{-1}xk)^y) ).
\end{align*}
In particular, we see the map $\Gal(\tilde{L}/Q)\ra G'$ given by $(\psi,z)$ has the same image, $G'$, as the map given by $(\rho,y)$.  

Now if $(\rho,y)$ is type $G'$, and $\rho': \Gal(L/K)\isom G$, then we can write $\rho'=\alpha \circ \rho$ for some $\alpha\in \Aut(G)$, and then
 $(\rho',y)$ is type $G'$ if and only if $\alpha$ fixes $G'$.
  So we have $|c|$ choices for the twist, and for each twist, exactly $\Aut_{G'}(G)$ choices for the map $\Gal(L/K)\isom G$ to get type $G'$.

\end{proof}

So when $G'$ is good
\begin{equation*}
E^+(G,G')=\frac{1}{|c||\Aut_{G'}(G)|}\tilde{E}^+(G,G').
\end{equation*}

\subsection{Requirement of admissibility}

Now we will see why the only possible types are the \sa $G'$.
\begin{proposition}\label{P:whysa}
Let $Q$ be $\Q$ or $\F_q(t)$, and 
let $G$ be a finite group.
If $K/Q$ is a quadratic extension with $\rho\in \Sur(G_K^{\un,\v}, G)$ such that $\rho$ (or $(\rho,y)$) is type $G'$,
then $G'$ is \sa.
\end{proposition}
\begin{proof}
By the definition of type, if $\rho$ (or $(\rho,y)$) is type $G'$,  then $G'$ contains $(1,\sigma)$ and the 
 kernel of $\pi: G' \ra S_2$ surjects onto the first factor of $G$. So if we suppose we have a type $G'$ that is not \sa, it must have that the order  $2$ elements of $G'\setminus \ker\pi$ do not generate $G$.
Since $\tilde{L}/K$ is unramified everywhere and split completely at $\v$, the only possible non-trivial images for inertia groups, or for the decomposition group at $\v$, are order  $2$ elements of $G'\setminus \ker\pi$.
  If these elements do not generate all of $G'$, they generate a proper subgroup $H$ of $G'$.  
  Since the set of  order  $2$ elements of $G'\setminus \ker\pi$ is closed under conjugacy, $H$ is normal.
By Galois theory, $H$ corresponds to a non-trivial extension of $Q$ that is unramified everywhere and split completely at $\infty$, which does not exist.
\end{proof}

\subsection{Reduced Schur multiplier, $H_2(G',c)$}\label{S:Hdefs}
In this section, we will define the \emph{reduced Schur multiplier} $H_2(G',c)$ of a finite group $G'$ for a union $c$ of conjugacy classes of $G'$, following
Venkatesh and Ellenberg \cite[Section 2.4]{Venkatesh2010}.
Let $\SC$ be a Schur cover of $G'$ so we have an exact sequence
$$
1\ra H_2(G',\Z)\ra \SC \ra G' \ra 1
$$
by the Schur covering map.
For $x,y\in G'$ that commute, let $\hat{x}$ and $\hat{y}$ be arbitrary lifts to $\SC$, and let $\langle x,y\rangle$
be the commutator $[\hat{x},\hat{y}]\in \SC$, which actually lies in 
$ H_2(G',\Z)$ since $x$ and $y$ commute.  
Alternatively, $\langle x,y\rangle$ is the image of a canonical generator of $H_2(\Z\times \Z,\Z)=\wedge^2 \Z^2$ in $H_2(G',\Z)$ under the map induced from the map $\Z\times \Z\ra G'$ taking $(1,0)$ to $x$ and $(0,1)$ to $y$.
  If we take the quotient of the above exact sequence by all $\langle x,y\rangle$ for $x\in c$ and $y$ commuting with $x$, we obtain an exact sequence 
\begin{equation}\label{E:defH}
1\ra H_2(G',c)\ra \widetilde{G'}_c \ra G' \ra 1,
\end{equation}
which is still a central extension, defining $H_2(G',c)$ and  $\widetilde{G'}_c$ . 
Note that since a Schur cover of a group $G'$ is not unique, these definitions implicitly depend on the choice of Schur cover. 
Throughout this paper we will just pick one Schur cover for each $G'$ and use that. 
However, the alternative description of  $\langle x,y\rangle$ shows that $H_2(G',c)$ does not depend on the choice of Schur cover.

\section{Lifting invariant}\label{S:lift}
The value $|H_2(G',c)[|\mu_{\F_q(t)}|]|$ appearing in Theorem~\ref{T:FF} comes from the number of components of a moduli space for the objects counted by $E^\pm_{\F_q(t),q^{2n}}(G,G')$.  Since in the analogy where $\F_q(t)$ is replaced by $\Q$, there is no analogous moduli space with components, in this section we give a different construction of the invariant that separates components.  Our construction works over any global field. This gives additional weight to using Theorem~\ref{T:FF} as a motivation for Conjecture~\ref{C:E}.  As discussed in the introduction, our construction is strongly motivated by that of Venkatesh and Ellenberg \cite[Section 2.5]{Venkatesh2010} but in order to deal with even $|G|$ (which are the main interest of this paper) we require a different construction than theirs, and we prove that our invariant actually lies in $|H_2(G',c)[|\mu_K|]|$ and not just $|H_2(G',c)|$.

 \subsection{Definition of the invariant}\label{S:defI}

First we will give the definition of the invariant as briefly as possible, and then we will prove that it is well-defined, does not depend on most of our choices, and behaves as we claim.

 \begin{notation}
For a field $K$, let $\mu_K$ be the group of roots of unity of $K$.  
Let  $G_K:=\Gal(\bar{K}/K)$.
For a positive integer $n$ and a field $K$, we write $K(\mu_n)$ to denote the minimal field extension of $K$ that includes all the $n$th roots of unity from some algebraic closure of $K$.
 \end{notation}

\begin{definition}\label{D:defI}
Let $F$ be a finite group, and $\pi: F \ra \Z/2\Z$ a surjective homomorphism.   Let $c$ be the set of order two elements of $F\setminus \ker \pi$, and \emph{suppose that $c$ is a single conjugacy class of $F$}. 
Let $K$ be a global field whose characteristic does not divide $|F|$, and let $\phi: G_K\ra F$ be a continuous, tame homomorphism such that all inertia has image in $c\cup \{1\}$.  (A tame homomorphism is one in which all wild inertia subgroups have trivial image, or equivalently the associated extension to the kernel of the homomorphism is not wildly ramified.) Then we define an invariant of $\phi$ as follows.
Choose a Schur cover of $F$ and let $\tilde{F}_c$ be as defined in Section~\ref{S:Hdefs} (implicitly depending on the Schur cover).
  Let $L=K(\mu_{4 |\tilde{F}_c|})$ and $\tilde{\phi}:  G_L \ra \tilde{F}_c$ be a tame lifting of $\phi$.  Let $u\in \mu_L$. 
For each non-archimedian place $v$ of $L$, let $\gamma_v\in G_L$ be an element of an inertia subgroup at $v$ that, in the maximal tame, abelian, exponent $|\tilde{F}_c|$ quotient of $G_L$, has the same image as $u$ (where $u$ maps to this quotient via the local class field theory isomorphism at $v$).
 Let $c_0$ be an element of $c$.  For $f\in \tilde{F}_c$ that maps to $c\cup \{1\}$, let $z(f)$ be the unique element
 in $\tilde{F}_c$ in the conjugacy class of $f$ that maps to either $1$ or $c_0$ in $F$.  Then we define the invariant as a product over finite places $v$ of $L$:
$$
I_{F,\pi,K}(\phi,u):=\prod_v z( \tilde{\phi}(\gamma_v)),
$$
which we will show below is in $H_2(F,c)[|\mu_K|]$.
\end{definition}

We will first show that the definition makes sense--that $\tilde{\phi}$ and $z$ exist.  Then, we will show that $I_{F,\pi,K}(\phi,u)$ does not depend on several choices we made: the order of multiplication, the choice of the $\gamma_v$,
the choice of $\tilde{\phi} $, and the choice of $c_0$.  Finally, we will show that $I_{F,\pi,K}(\phi,u)\in H_2(F,c)[|\mu_K|]$ and that the map
\begin{align*}
\mu_L &\ra H_2(F,c)[|\mu_K|]\\
u &\mapsto I_{F,\pi,K}(\phi,u)
\end{align*}
is a homomorphism.
So, given $F,\pi,$ and $K$, we can take $I_{F,\pi,K}$ to define an invariant of
continuous, tame $\phi: G_K\ra F$ valued in $\Hom(\mu_L, H_2(F,c)[|\mu_K|] )$.  Clearly
there are $|H_2(F,c)[|\mu_K|]|$ values for this invariant.
We do not show that the invariant is independent of choice of the choice of Schur cover, though we see that the number of possible invariants is independent of this choice.

First we show that the tame lift $\tilde{\phi}$ exists.  

\begin{lemma}
Let $F$ be a finite group, and $c$ a union of conjugacy classes of order $2$ elements of $F$.
Let $L$ be a global field such that $4\mid |\mu_L|$ and for every element $f\in \tilde{F}_c$ we have
$f^{|\mu_L|}=1$.
Let $\phi:G_{L}\ra F$ be a continuous, tame homomorphism such that all inertia groups have image in $c \cup \{1\}$.
There is a tamely ramified lift $\tilde{\phi}:G_{L}\ra \tilde{F}_c$ of $\phi:G_{L}\ra F.$
\end{lemma}

\begin{proof}
Let $A=H_2(F,c)=\ker(\tilde{F}_c\ra F)$. 
The obstruction to lifting $\phi$ to $\tilde{\phi}$ lies in $H^2(G_{L},A)$
\cite[I.5.7, Prop. 43]{Serre2002}, i.e. $\phi\in \Hom(G_{L},F)$ is in the image of $ \Hom(G_{L},\tilde{F}_c)$ if and only if it maps to $0$ in $H^2(G_{L},A)$.
The kernel of $H^2(G_{L},A)\ra \prod_v H^2(G_{L_v},A)$
is dual to the kernel of
$H^1(G_L,A^*)\ra \prod_v H^1(G_{L_v},A^*)$ \cite[Chapter I, Theorem 4.10]{Milne2006}, where $A^*:=\Hom(A,\bar{L})$
and the products are over places $v$ of $L$.
Since $L$ contains enough roots of unity to annihilate $A$, the $G_L$ action on $A^*$ is trivial and we have that $A\isom A^*$ as 
$G_{L}$ modules.
However, the map $H^1(G_L,A^*)\ra \prod_v H^1(G_{L_v},A^*)$ is clearly injective, since an element in the kernel would give an extension of $L$ split completely everywhere.

Thus, we can detect if $\phi$ lifts to some $\tilde{\phi}$ by considering the obstruction in
$$
\prod_v H^2(G_{L_v},A).
$$
So for each place $v$ of $L$, we will show we can lift the map $\phi$ when restricted to $G_{L_v}$, i.e. that the local obstruction is trivial.  Since $\phi$ is tamely ramified, we consider standard generators $x,y$ of the tame quotient of $G_{L_v}$ such that
$xyx^{-1}=y^{\Nm v}$.  
Let $Y$ be any lift in $\tilde{F}_c$ of $\phi(y)$, and $X\in \tilde{F}_c$ any lift of $\phi(x)$.  At unramified places, we can choose a lift of $\phi:G_{L_v}\ra F$ with $y\mapsto 1$ and $x \mapsto X\in\tilde{F}_c$.

At ramified places, since all ramification is order $2$, we must have $\Nm v$ is odd since the ramification is tame.
In this case, $\phi(y)\in c$ implies $\phi(y)$ has order $2$ and so commutes with $\phi(x)$, and 
by the definition of $\tilde{F}_c$ any lift of $\phi(y)$ to $\tilde{F}_c$ commutes with any lift of an element commuting with $\phi(x)$ to $\tilde{F}_c$.  Thus, $XYX^{-1}=Y$.  Let $p_v$ be the prime of which $\Nm v$ is a power.  Let $d$ be the maximum power of $p_v$ dividing the order of $Y$.  Let $Z=Y^d$.
Note that $XZX^{-1}=Z$.
For some non-negative integer $k$, we have $|\mu_L|\mid p_v^k(\Nm v -1)$.  
Since $Y^{|\mu_L|}=1$, we have  $Z^{\Nm v -1}=1$.
 Thus we can let $x\mapsto X$ and $y\mapsto Z$ to give a map from the tame quotient of  $G_{L_v}$ to $\tilde{F}_c$ lifting $\phi$.  This proves that there is some lift $\tilde{\phi}:G_{L}\ra \tilde{F}_c$ of $\phi:G_{L}\ra F.$
 We conclude there is some global lift of $\phi$ to a map $\tilde{\phi}: G_L \ra \tilde{F}_c$.
 
When $L$ is a function field, the hypothesis that
for every element $f\in \tilde{F}_c$ we have
$f^{|\mu_L|}=1$ implies that the characteristic of $L$ does not divide $|\tilde{F}_c|$ and $\tilde{\phi}$ must then be tamely ramified.
  When $L$ is a number field,
 we know at each place $v$ of $L$, there is a tamely ramified local lift $\psi_v$ of $\phi_v :G_{L_v}\ra F$ to $\tilde{F}_c$ by the construction above.  Note that $\tilde{\phi}^{-1} \cdot \psi_v$ gives a map from $ G_{L_v}$ to $A$.  Since $\mu_4\sub L$, by the Grunwald-Wang Theorem \cite[Ch 10. Thm. 5]{Artin1968}, there is a map $\pi: G_{L} \ra A$ which agrees with $\tilde{\phi}^{-1} \psi_v$ on $ G_{L_v}$ for each  of the finitely many places $v$ of $\tilde{K}$ whose norms are not relatively prime to $|\tilde{F}_c|$. 
  Then $\tilde{\phi} \cdot \pi$ is a map from $G_{L}$ to $\tilde{F}_c$, which at each place of $L$ that could possibly be wildly ramified, agrees with the tamely ramified map $\psi_v: G_{L_v}\ra \tilde{F}_c$.  The lemma follows.
\end{proof}

Next we show that the function $z$ used in Definition~\ref{D:defI} exists.
\begin{lemma}\label{L:conjabove}
Let $F$ be a finite group and $c$ a union of conjugacy classes of $F$.
Any conjugacy class of $\tilde{F}_c$ with image in $c\cup \{1\}\sub F$ is in bijection with its image in $F$.
\end{lemma}
\begin{proof}
An element above $1\in F$ is in the center of $\tilde{F}_c$, and so its conjugacy class is a single element and the lemma follows in this case.  Let $\bar{x}$ denote the image in $F$ of an element $x\in \tilde{F}_c$.
Let $d\in \tilde{F}_c$ with $\bar{d}\in c$.  Suppose $g\in \tilde{F}_c$ is such that $\bar{d}=\overline{gdg^{-1}}\in F$.
Then $\bar{g}$ commutes with $\bar{d}$ in $F$, and $\bar{d}\in c$ implies $g$ and $d$ commute in $\tilde{F}_c.$
Thus the lemma follows.
\end{proof}

Next we show that the order of the product in Definition~\ref{D:defI} does not matter.
\begin{lemma}
In Definition~\ref{D:defI}, the elements $z( \tilde{\phi}(\gamma_v))$ in the  product defining $I_{F,\pi,K}(\phi,u)$ all commute.
\end{lemma}
\begin{proof}
All the $z( \tilde{\phi}(\gamma_v))$ are either in $H_2(F,c)$, and thus central, or above $c_0$.
All elements above $c_0$ commute (since they only differ from one another by central elements).
\end{proof}

Next we show that the invariant defined does not depend on the choice of $\gamma_v$.
\begin{lemma}\label{L:anyg}
In Definition~\ref{D:defI}, the invariant $I_{F,\pi,K}(\phi,u)$ does not depend on the choice of $\gamma_v$.
\end{lemma}

\begin{proof}
Let $y_v$ be a generator of tame inertia for $v$ in $G_L$. 
 Then $\gamma_v$ has image $g y_v^\lambda g^{-1}$ in the tame quotient of $G_L$ for some $\lambda\in \hat{\Z}$ and $g\in G_L$.  Suppose we take another choice $\gamma'_v$ with image $h y_v^{\lambda'} h^{-1}$ 
 in the tame quotient of $G_L$ for some $h$ in $G_L$ in Definition~\ref{D:defI}.  In the tame, abelian, exponent $|\tilde{F}_c|$ quotient, $\gamma_v$ and $\gamma'_v$ have the same image (corresponding to $u$ in the class field theory map) and thus in this quotient  $y_v^\lambda\equiv y_v^{\lambda'}$.  That implies $\lambda\equiv \lambda'$ mod $(\Nm v -1,|\tilde{F}_c|)$.  If $p$ is the prime of which $\Nm v$ is a power, and $d$ is the highest power of $p$ dividing $|\mu_L|$, then $|\mu_L|/d \mid (\Nm v -1).$
Since $\tilde{\phi}$ is tame, it sends any inertia group for $v$ to a subgroup whose order is relatively prime to $p$, and in particular, a subgroup whose elements have order dividing $|\mu_L|/d$.  Thus $\tilde{\phi}(y_v^\lambda)=\tilde{\phi}(y_v^{\lambda'})$.  Since $I_{F,\pi,K}(\phi,u)$ only depends on the conjugacy class of $\tilde{\phi}(\gamma_v)$, the lemma follows.
\end{proof}

We show that the invariant defined does not depend on the choice of lift $\tilde{\phi}$.

\begin{lemma}\label{L:anyphi}
In Definition~\ref{D:defI}, the invariant  $I_{F,\pi,K}(\phi,u)$ does not depend on the choice of $\tilde{\phi}$.
\end{lemma}

\begin{proof}

Let $A=H_2(F,c)=\ker(\tilde{F}_c\ra F)$.
A different choice of $\tilde{\phi}$ would differ by $\chi: G_{L} \ra A$.
From Lemma~\ref{L:anyg}, we may choose $\gamma_v$ for each non-archimedian place $v$ corresponding to $u$ in the local class field theory map.
Then for each finite place $v$ of $L$, we have that $\tilde{\phi}(\gamma_v)\chi(\gamma_v)$ is in a conjugacy class over $c\cup \{1\}$ and the element above $c_0$ or $1$ is
$z(\tilde{\phi}(\gamma_v))\chi(\gamma_v)$. 
 Since the $\gamma_v$ all come from $u$ in the local class field theory maps, by class field theory, we have $\prod_v \chi(\gamma_v)=1$.
\end{proof}

Now, we show that the invariant actually lands in $H_2(F,c)$.

\begin{lemma}\label{L:inH}
In Definition~\ref{D:defI}, $I_{F,\pi,K}(\phi,u)$ is in $\ker(\tilde{F}_c\ra F)$.
\end{lemma}
\begin{proof}
Each $z( \tilde{\phi}(\gamma_v))$ is either a lift of $c_0$ or in $H_2(F,c)$.
It suffices to show that the number of $v$ such that $z_v:=z( \tilde{\phi}(\gamma_v))$ is a lift of $c_0$ is even.

If $v$ is even, then since $\phi$ has all ramification order $2$ and is tame, then $v$ is unramified in $\phi$, and so $z_v$ cannot be a lift of $c_0$.  For the rest of the proof we consider odd places $v$.

First we consider $v$ such that $u$ is a square mod $v$.  Then by Hensel's lemma $L_v$ contains $w$ such that $w^2=u$, and let $\delta$ be an element of inertia at $v$ that, after abelianization, goes to $w$ in the class field theory map.  Then, using Lemma~\ref{L:anyg}, we can let $\gamma_v=\delta^2$, and we see that $z_v$ is a square and therefore is a lift of an element of the index $2$ subgroup $\ker \pi$ of $F$, and thus cannot be a lift of an element of $c_0$.

Finally, we consider $v$ such that $u$ is not a square mod $v$.  If $\phi$ is not ramified at $v$ then $z_v$ is not a lift of $c_0$.  What remains are the $v$ such that $\phi$ is ramified at $v$ and $u$ is not a square mod $v$.
Consider the composition $G_L \ra F \stackrel{\pi}{\ra} \Z/2\Z$, and the corresponding map $J_L \ra \Z/2\Z$ of the 
id\`ele class group of $L$.  Since $u$ goes to $0$ in this map, the number of places such that $u\mapsto 1$ in the corresponding local map $L_v \ra \Z/2\Z$ must be even.  We have that $u$ goes to $1$ in the map $L_v \ra \Z/2\Z$
exactly if $v$ is ramified in $\phi$, and $u$ is not a square mod $v$.  This proves the lemma.
\end{proof}

Next, we show that the invariant defined does not depend on $c_0$.
\begin{lemma}
In Definition~\ref{D:defI}, $I_{F,\pi,K}(\phi,u)$ does not depend on the choice of $c_0$.
\end{lemma}
\begin{proof}
 If we chose $c_1$ instead of $c_0$, we would replace each $z( \tilde{\phi}(\gamma_v))$ with $gz( \tilde{\phi}(\gamma_v)) g^{-1}$ for some $g$ such that $gc_0g^{-1}=c_1$.  Thus, $I_{F,\pi,K}(\phi,u)$ would be replaced by $gI_{F,\pi,K}(\phi,u)g^{-1},$ which equals $I_{F,\pi,K}(\phi,u)$ since $I_{F,\pi,K}(\phi,u)\in H_2(F,c)$ by Lemma~\ref{L:inH}
 and $H_2(F,c)$ is central in $\tilde{F}_c$.
\end{proof}

\begin{lemma}\label{L:phiconj}
With the notation of Definition~\ref{D:defI}, if  $f\in F$, then  $I_{F,\pi,K}(\phi,u)=I_{F,\pi,K}(f\phi f^{-1},u)$ 
\end{lemma}
\begin{proof}
Let $\psi=f\phi f^{-1}$.  Let $\hat{f}$ be a lift of $f$ to $\tilde{F}_c$.  Then $\tilde{\psi}=\hat{f}\tilde{\phi} \hat{f}^{-1}$
is a tamely ramified lift of $\psi$ to $\tilde{F}_c$.  Since $\hat{f}\tilde{\phi} \hat{f}^{-1}(\gamma_v)$ is in the same conjugacy class as $\tilde{\phi} (\gamma_v)$, the lemma follows.
\end{proof}

Finally, we show that the invariant is $|\mu_K|$-torsion.
\begin{lemma}
In Definition~\ref{D:defI}, $I_{F,\pi,K}(\phi,u)^{|\mu_K|}=1$.
\end{lemma}
\begin{proof}
Let $\sigma\in G_K$.  Let $\psi(x):=\phi(\sigma x \sigma^{-1})$.  Since $\phi$ is defined over $G_K$, we have that 
$\psi(x)=\phi(\sigma) \phi(x) \phi(\sigma)^{-1}$ and thus by Lemma~\ref{L:phiconj} $I_{F,\pi,K}(\psi,u)=I_{F,\pi,K}(\phi,u)$.
However, note that $\tilde{\psi}(x):=\tilde{\phi}(\sigma x \sigma^{-1})$ gives a lift of $\psi$.  
We choose $\gamma_v\in G_L$ whose image in the abelianization correspond to $u$ in the local class field theory map.
We have $\tilde{\psi}(\gamma_v)=\tilde{\phi}(\sigma \gamma_v \sigma^{-1})$.  

We have the commutative diagram
$$
\begin{CD}
L_v^* @>{\isom}>> G_{L_v}^{ab}\\
@VV{\sigma}V @VV{\sigma}V\\
L_{\sigma(v)}^* @>{\isom}>> G_{L_{\sigma(v)}}^{ab}
\end{CD}
$$
by the functoriality of the local class field theory map.  Here the $\sigma$ on the left acts as $\sigma$ on elements of $L$ (and so takes Cauchy sequences for the $v$-adic metric to Cauchy sequences for the $\sigma(v)$-adic metric).
The top and bottom maps are the local class field theory maps.  The right hand map is the map that takes $f$ to $\sigma \circ f \circ \sigma^{-1}$.
Note that $\sigma \gamma_v \sigma^{-1}$ is an element of tame inertia for $\sigma(v)$ and its image in the abelianization  comes from  $u^{\chi(\sigma)}$ in the local class field theory map by the commutativity of the above diagram, where $\chi$ is the cyclotomic character.  

So, we could choose $\sigma \gamma_v^{({\chi(\sigma)}^{-1})} \sigma^{-1}$ for $\gamma'_{\sigma(v)},$ since they have image in the abelianization coming from $u$ in the class field theory map.
So in the conjugacy class of $\tilde{\phi}(\gamma_{\sigma(v)})$ we also have 
$\tilde{\phi}(\sigma \gamma_v^{{\chi(\sigma)}^{-1}} \sigma^{-1})$ by Lemma~\ref{L:anyg}.

Thus $z(\tilde{\psi}(\gamma_v)^{{\chi(\sigma)}^{-1}})=z(\tilde{\phi}(\gamma_{\sigma(v)})).$  Multiply over all $v$, we obtain $I_{F,\pi,K}(\psi,u)^{{\chi(\sigma)}^{-1}}=I_{F,\pi,K}(\phi,u).$
So $I_{F,\pi,K}(\psi,u)=I_{F,\pi,K}(\phi,u)^{{\chi(\sigma)}}.$
Thus we conclude $I_{F,\pi,K}(\phi,u)^{{\chi(\sigma)-1}}=1$ for all $\sigma\in G_K$.  For the $\sigma\in G_K$, the great common divisor of
$\chi(\sigma)-1$ is $|\mu_K|$. (Proof: Clearly $\chi(\sigma)$ is $1$ mod $|\mu_K|$ for all $\sigma$.  Moreover if $\chi(\sigma)$ is $1$ mod $d$ for all $\sigma\in G_K$, then $G_K$ fixes the $d$th roots of unity, and thus $\mu_d\sub K$.  ) The lemma follows.
\end{proof}

\begin{lemma}
In Definition~\ref{D:defI}, $I_{F,\pi,K}(\phi,u^\lambda)=I_{F,\pi,K}(\phi,u)^\lambda$.
\end{lemma}
\begin{proof}
 If we replace $u$ with $u^\lambda$ for some $\lambda\in \hat{\Z}$, we see that we can replace $\gamma_v$ with $\gamma_v^\lambda$.  The lemma follows.
\end{proof}

 \subsection{Agreement with the topological invariant in function field case}\label{SS:same}

In this section, we will see in Theorem~\ref{T:same} that Definition~\ref{D:defI} essentially agrees with the lifting invariant defined by Ellenberg, Venkatesh, and Westerland \cite[Section 8.4]{Ellenberg2012}, which counts components in the Hurwitz schemes used to prove Theorem~\ref{T:FF}. 

 Ellenberg, Venkatesh, and Westerland define the lifting invariant for continuous, tame homomorphisms $G_{\bar{\F}_q(t)} \ra F$ when the characteristic of $\F_q$ does not divide $|F|$.  The invariant can then also be defined for  
$G_{\F_q(t)} \ra F$, by restriction to $G_{\bar{\F}_q(t)}  \sub G_{\F_q(t)}$ which amounts to base change to the cover of $\P^1$ associated to the homomorphism from $\F_q$ to $\bar{\F}_q$. 
This lifting invariant tells you what component of the Hurwitz space an extension is on (see Theorem~\ref{T:Chur} \eqref{i:comp}).
 We now give the definition of this lifting invariant.

\begin{definition}[from Section 8.4 of \cite{Ellenberg2012}]\label{D:EVW}
Let $F$ be a finite group and $p$ a prime not dividing $|F|$.  
 Let $c$ be a union of conjugacy classes of $F$.
 Let $\phi: G_{\bar{\F}_p(t)} \ra F$ 
be a  continuous, tame homomorphism, with inertia subgroups with images only in $c\cup \{1\}$ and let $S$ 
be the union of $\{\infty\}$ and the set of $\bar{\F}_p$ points of the corresponding $\P^1$ where $\phi$
is ramified.  Let $\Gamma$ be the tame quotient of $G_{\bar{\F}_p(t)}$ corresponding to extensions only 
ramified at $S$.  Then $\Gamma$ is free pro-prime-to-$p$, on generators $\delta_1,\dots,\delta_k$, with each generator generating the tame inertia subgroup at a point of $S\setminus \infty$.
We can choose the $\delta_i$ so that
$$
\delta_1  \delta_2 \cdots \delta_k \delta_\infty=1
$$
in $\Gamma$, where  $\delta_\infty$ is a generator of tame inertia at $\infty$.
(This is possible by Grothendieck's comparison of \'etale and topological $\pi_1$ \cite[Corollaire 2.12, Expos\'e XIII]{Grothendieck2003}.)
 For each conjugacy class in $c$, we choose a conjugacy class in $\tilde{F}_c$ above it, and for $f\in c$, let $[f]\in \tilde{F}_c$ denote its preimage in the chosen conjugacy class (see Lemma~\ref{L:conjabove}). 
Then, the invariant of $\phi$ is
$$
J_{F,c,p}(\phi,\delta_1,\dots,\delta_k)=[\phi(\delta_1)][\phi(\delta_2)]\cdots [\phi(\delta_k)]\in \tilde{F}_c.
$$
\end{definition}

Since, at one time, Theorem~\ref{T:FF} only considers $\phi$ with $\phi(\delta_\infty)$ some fixed value, for our purposes the following theorem shows agreement of our invariant over $\F_q(t)$ with that of Ellenberg, Venkatesh, and Westerland.

\begin{theorem}[Agreement of invariants]\label{T:same}
Let $F$ be a finite group and $q$ a power of a prime $p$ not dividing $|F|$.  
Let $\pi: F \ra \Z/2\Z$ be a surjective homomorphism.   Let $c$ be the set of order two elements of $F\setminus \ker \pi$, and suppose that $c$ is a single conjugacy class of $F$. 
Let $\phi: G_{\F_q(t)} \ra F$  be a  continuous, tame homomorphism, with inertia subgroups with images only in $c\cup \{1\}$.  
  Given $\delta_1,\dots,\delta_k$ (and $\delta_\infty$ if $\infty$ is ramified in $\phi$) chosen in Definition~\ref{D:EVW} for $\phi|_{G_{\bar{\F}_p(t)}}$, there exists a root of unity $u$ in $L$ (from Definition~\ref{D:defI} for $K=\F_q(t)$) such that
  $$
  J_{F,c,p}(\phi|_{G_{\bar{\F}_p(t)}},\delta_1,\dots,\delta_k)=[c_0]^m I_{F,\pi,\F_q(t)}(\phi,u)
 [\phi(\delta_\infty)]^{-1},
  $$
where $m$ is the number of $\bar{\F}_p$ points ramified in $\phi$, $[\cdot]$ is as in Definition~\ref{D:EVW} and $[1]=1$, and $c_0$ is any element of $c$.
\end{theorem}

Whenever we are given a prime power $q$, following \cite[Section 8.1]{Ellenberg2012} we define $\hat{\Z}$ to be the ring $\lim\limits_{\leftarrow}\Z/n\Z$  and $\hat{\Z}(1)$ to be the group $\lim\limits_{\leftarrow} \mu_n$, where $\mu_n$ denotes the group of $n$th roots of unity in an algebraic closure of $\bar{\F}_q$ and
both limits are taken over $n$ relatively prime to $q$
 (and thus $\hat{\Z}$ and $\hat{\Z}(1)$ will always implicitly depend on a chosen $q$). 
 We write $\hat{\Z}(1)^\times$ for the set of topological generators of $\hat{\Z}(1)$, which carries a simply transitive action of $\hat{\Z}^\times$.
  If $X$ is a set with an action of $\hat{\Z}^\times$, we write
 $$
 X\langle -1 \rangle :=\operatorname{Mor}_{\hat{\Z}^\times} (\hat{\Z}(1)^\times,X)
 $$
 for the set of functions $\hat{\Z}(1)^\times\ra X$ equivariant for the $\hat{\Z}^\times$ actions.  If we fix an element $\underline{u}$ of $\hat{\Z}(1)^\times$,
 elements of  $X\langle -1 \rangle$ correspond exactly to elements of $X$ (given by taking the image of $\underline{u}$).
 
 The $\delta_1,\dots,\delta_k$ above determine an element $\underline{u}$ of $\hat{\Z}(1)^\times$, which we will describe in the course of the proof of Theorem~\ref{T:same}.  We write $\Z^{c/F}$ for the free abelian group on the set of conjugacy classes in $c$.  
 Given  $\phi: G_{\bar{\F}_p(t)} \ra F$, we have an element $N(\phi)\in \Z^{c/F}$ in which the coordinate in each component gives the number of $\bar{\F}_p$ points (other than infinity) where $\phi$ has inertia of that conjugacy class.  We have a map $\Z^{c/F}\ra F^{ab}$ taking a generator for each conjugacy class the an image of an element in that conjugacy class.
  If we have an action of $\hat{\Z}^\times$ on the set $\tilde{F}_c \times_{F^{ab}}\Z^{c/F}$ (see Equation~\eqref{E:discrete} where we will define such an action), 
  then we get an element of $(\tilde{F}_c \times_{F^{ab}}\Z^{c/F}) \langle -1 \rangle$ by taking $\underline{u}$ to $$(J_{F,c,p}(\phi|_{G_{\bar{\F}_p(t)}},\delta_1,\dots,\delta_k),N(\phi))\in \tilde{F}_c \times_{F^{ab}}\Z^{c/F}.$$  
(See \cite[Proposition 8.5.2]{Ellenberg2012} for a proof that the element of   $(\tilde{F}_c \times_{F^{ab}}\Z^{c/F}) \langle -1 \rangle$ does not depend on the choice of the $\delta_i$.)
  Given $F,c,p$, we call this element of $(\tilde{F}_c \times_{F^{ab}}\Z^{c/F}) \langle -1 \rangle$ the \emph{component invariant} of $\phi$ (for reasons that will become clear in the statement of Theorem~\ref{T:Chur}). 

\begin{proof}[Proof of Theorem~\ref{T:same}]
We use the notations from Definitions~\ref{D:defI} and \ref{D:EVW}.
Since $\Gamma$ is free, we can lift $\phi: G_{\bar{\F}_p(t)} \ra F$ to a $\hat{\phi}: G_{\bar{\F}_p(t)} \ra \tilde{F}_c$ that is tame and ramified only at points in $S$.  Moreover, since $\hat{\phi}$  corresponds to a finite cover of $\P^1$, it descends to $\F_{q^r}$ for some positive integer $r$, and thus for $M=\F_{q^r}(t)$ we have that $\hat{\phi}$ extends to $\hat{\phi}: G_{M} \ra \tilde{F}_c$.
We enlarge $r$ as necessary so that all the points of $S$ are $\F_{q^r}$ points, and so that $L=\F_{q}(t)(\mu_{4|\tilde{F}_c|})$  is a subfield of $\F_{q^r}(t).$
Further, since $\Gamma$ is free pro-prime-to-$p$ on $\delta_1,\dots,\delta_k$, for each $1\leq i \leq k$, we have a continuous map $G_{\bar{\F}_p(t)}\ra \Z/4|\tilde{F}_c|\Z$ that sends $\delta_i$ to $1$ and $\delta_j$ to $0$ for $j\neq i$.  Enlarge $r$ further so that all of the maps of curves over $\bar{\F}_p$ corresponding to the above homomorphisms $G_{\bar{\F}_p(t)}\ra \Z/4|\tilde{F}_c|\Z$ descend to $\F_{q^r}$, i.e. so that
the maps of Galois groups extend to $G_M \ra \Z/4|\tilde{F}_c|\Z$.

We have
  $$
[\phi(\delta_1)][\phi(\delta_2)]\cdots [\phi(\delta_\infty)]=\hat{\phi}(\delta_1)\hat{\phi}(\delta_1)^{-1}[\phi(\delta_1)]
\hat{\phi}(\delta_2)\hat{\phi}(\delta_2)^{-1}[\phi(\delta_2)]\cdots
.
$$
Since $\hat{\phi}(\delta_i)^{-1}[\phi(\delta_i)]$ has trivial image in $F$, it is central in $\tilde{F}_c$.  Thus, we can factor the 
$\hat{\phi}(\delta_i)$ terms to the left and obtain
\begin{align*}
[\phi(\delta_1)][\phi(\delta_2)]\cdots [\phi(\delta_\infty)]&=\hat{\phi}(\delta_1\delta_2\cdots)\hat{\phi}(\delta_1)^{-1}[\phi(\delta_1)]
\hat{\phi}(\delta_2)^{-1}[\phi(\delta_2)]\cdots\\
&=\hat{\phi}(\delta_1)^{-1}[\phi(\delta_1)]
\hat{\phi}(\delta_2)^{-1}[\phi(\delta_2)]\cdots
.
\end{align*}
Let $g_i [\phi(\delta_i)] g_i^{-1}$ be above $c_0\in c$ (which, if $i=\infty$ only exists if $\infty$ is ramified in $\phi$). Then by the property of $[\cdot ]$, we have
$g_i [\phi(\delta_i)] g_i^{-1}=[c_0].$
Since $\hat{\phi}(\delta_i)^{-1}[\phi(\delta_i)]$  is central, 
$$
\hat{\phi}(\delta_i)^{-1}[\phi(\delta_i)]=
g_i \hat{\phi}(\delta_i)^{-1} g_i^{-1}
g_i [\phi(\delta_i)] g_i^{-1}=z( \hat{\phi}(\delta_i))^{-1} [c_0].
$$
Now if $\infty$ is not ramified in $\phi$, 
$$
\hat{\phi}(\delta_\infty)^{-1}[\phi(\delta_\infty)]=z(\hat{\phi}(\delta_\infty))^{-1}.
$$
Since $z( \hat{\phi}(\delta_i))^{-1}$ and $[c_0]$ all have image $c_0$ in $F$, they all commute and we have
\begin{equation}\label{E:s1}
[\phi(\delta_1)][\phi(\delta_2)]\cdots [\phi(\delta_\infty)]=[c_0]^m\prod_i z( \hat{\phi}(\delta_i))^{-1},
\end{equation}
where $m$ is the number of ramified $\bar{\F}_p$ points of $\phi$.

We will see, with some work, that $\prod_i z( \hat{\phi}(\delta_i))$ is the invariant defined in Definition~\ref{D:defI}.  It is not immediate because we are working over a field larger than $L$ and the $\delta_i$ are picked differently than the $\gamma_v$.

Now in any tame, abelian extension of $M$ unramified outside $S$, we have that $\delta_1\cdots \delta_\infty=1$ in the Galois group.  What is the Galois group of the maximal tame, abelian, unramified outside $S$ extension of $M$?
By class field theory, it is the profinite completion $\hat{\Delta}$ of
$$
\Delta:=\left(\prod_{w\in S} M_w^*/U_w \times \prod_{w\not\in S}M_w^*/\O_{M_w}^* \right)/M^*,
$$
where $U_w$ are the units that are $1$ mod $w$.  
The image of $\delta_i$ in  $\hat{\Delta}$ is a generator for the inertia subgroup for the point corresponding to $i$, and thus can be represented in $\hat{\Delta}$ by $(1,\dots,1,\alpha_i,1\dots )$, where the $\alpha_i$ is in the place corresponding to $i$ and $\alpha_i$ generates the units in that component.
Since the image of $\delta_1\cdots \delta_\infty$ in $\hat{\Delta}$ is trivial,
 there is some $\alpha\in M$, such that for all $i$ we have that $\alpha\equiv \alpha_i$ mod $w$, where $w$ the  place corresponding to $i$.  
 
 Let $\zeta$ generate the roots of unity of $M$.
 For each $i$ corresponding to a place $w\in S$, let $\zeta=\alpha_i^{e_i}$ mod $w$.  Then since $\zeta$ has trivial image in $\hat{\Delta}$, so does
$\prod_{i} \delta_i^{e_i}$.    
However, since for each $i\ne\infty$, there is a map that sends $\delta_i$ to $1$ and $\delta_\infty$ to $-1$ in $\Z/|\tilde{F}_c|\Z$ and the rest of the $\delta_{j}$ to $0$ for $j\neq i$, we have $e_i=e_\infty$ mod $|\tilde{F}_c|$ for all $i$.  Thus, if $\xi$ is the root of unity in $M$ with image $\alpha_\infty$ mod $\infty$, then for each $i$, we have $\alpha_i\equiv \xi^{a_i}$ for some $a_i\equiv 1$ mod $|\tilde{F}_c|$. Note that $\xi$ generates the roots of unity in $M$ since $\alpha_\infty$ generates the local units.  

Note that if we replace $r$ above with a multiple of $r$, and thus replace $M$ with $M'$ containing $M$, it follows from class field theory that
the chosen root of unity $\eta'$ will have $\Nm_{M'/M} \eta'=\eta$.  Since $\lim\limits_{\leftarrow} \mu_{\F_q^r} = \Z(1)$ (with the maps on the left-hand side being the norm maps), the roots of unity $\eta$ and all possible $\eta'$ determine an element $\underline{u}$ of $\Z(1)^\times$.

As in Definition~\ref{D:defI}, we pick a $\tilde{\phi}:G_L \ra \tilde{F}_c$, which restricts to $G_{M}$.
Analogous to Definition~\ref{D:defI}, except using $M$ in place of $L$, for each place $w$ of $M$, we pick $\gamma_w$ inertia elements at $w$ that in the maximal tame, abelian, exponent $|\tilde{F}_c|$ quotient has image $\xi$ in the local class field theory map for $w$.  So we have $z(\hat{\phi}(\delta_w))=z(\hat{\phi}(\gamma_w))$, by the same reasoning as in Lemma~\ref{L:anyg}.
Thus, using \eqref{E:s1}, we have 
$$
[\phi(\delta_1)][\phi(\delta_2)]\cdots [\phi(\delta_\infty)]=[c_0]^m\prod_{w\in S} z( \hat{\phi}(\gamma_w))^{-1}. 
$$
By the same argument as in Lemma~\ref{L:anyphi},
$$
\prod_{w\in S} z( \hat{\phi}(\gamma_w))=\prod_{w} z( \tilde{\phi}(\gamma_w))
$$
where the second product is over all places of $M$.  (Recall $\hat{\phi}$ is only ramified at places in $S$.)
Now consider a place $v$ of $L$ and the set $T$ of places $w$ of $M$ over $v$.  We pick a $w_0\in T$,
and $\sigma\in G_L$ be a lift of a generator of $\Gal(M/L)$.  Let $\chi$ be the cyclotomic character, and we can choose $\sigma$ such that $\chi(\sigma)=q^{-a}$, where $\F_{q^a}$ is the finite subfield of $L$.  
For each $w\in T$ we have some $i_w$ such that  $\sigma^{i_w} (w_0)=w$, and the set of $i_w$ for $w\in T$ is $0,1,\dots,|T|-1$.  Then
$\sigma^{i_w}\gamma_{w_0} \sigma^{-i_w}$ is an element of inertia at $w$ that goes to $\sigma (\xi)=\xi^{q^{-{ai_w}}}$ in the class field theory map.  So
$\sigma^{i_w} \gamma_{w_0}^{q^{ai_w}} \sigma^{-i_w}$ is an element of inertia at $w$ that goes to $\xi$ in the class field theory map.
 Then, as in Lemma~\ref{L:anyg}, we have 
$$
z( \tilde{\phi}(\gamma_w))=z( \tilde{\phi}( \sigma_w \gamma_{w_0}^{{q^{ai_w}}} \sigma_w^{-1} ))
=z( \tilde{\phi}( \gamma_{w_0}^{{q^{ai_w}}}  )),
$$
since $\tilde{\phi}$ is defined on $G_L$.
Thus,
$$
\prod_{w\in T}  z( \hat{\phi}(\gamma_w))=z( \tilde{\phi}( \gamma_{w_0}^{\frac{q^{a|T|}-1}{q^a-1}}  )).
$$
Now $\gamma_{w_0}^{\frac{q^{a|T|}-1}{q^a-1}}\in G_L$ is an inertia element for $v,$ and by the functoriality of the local class field theory map, it has image corresponding to $\Nm_{M_{w_0}/L_{w_0}} \xi^{{\frac{q^{a|T|}-1}{q^a-1}}}=\Nm_{M/L} \xi$.  Note in any extension $\F_{q^a}(t)/\F_{q^b}(t)$ the norm acts on roots of unity by powering by $(q^a-1)/(q^b-1)$. So
$$
\prod_{w\in S} z( \hat{\phi}(\gamma_w))=I_{F,\pi,\F_q(t)}(\phi,\xi^{|\mu_M|/|\mu_L|}).
$$
So, 
$$
[\phi(\delta_1)][\phi(\delta_2)]\cdots [\phi(\delta_k)]=[c_0]^m I_{F,\pi,\F_q(t)}(\phi,\xi^{|\mu_M|/|\mu_L|})^{-1}
 [\phi(\delta_\infty)]^{-1}, 
$$
which proves the theorem.
\end{proof}

\section{Proof of Theorem~\ref{T:FF}}\label{S:FF}

In this section we prove a strengthening of Theorem~\ref{T:FF}.  Each surjection counted in the numerator of $E^{\pm}_{\F_q(t),q^{2n}}(G,G')$ has an associated extension $\tilde{L}/\F_q(t)$.  These field extensions correspond to covers of the projective line, and Hurwitz spaces provide a moduli space for such covers.  By the Grothendieck-Lefschetz trace formula, if we can control the dimensions of the \'{e}tale cohomology uniformly in $q$, then the question of the $q$ limit in Theorem~\ref{T:FF} is reduced to determining the number of components in the relevant moduli space.

Ellenberg, Venkatesh, and Westerland \cite{Ellenberg2012}, building on the work of Romagny and Wewers \cite{Romagny2006}, have constructed the Hurwitz schemes necessary for our application.  Moreover, 
Ellenberg, Venkatesh, and Westerland \cite{Ellenberg2012} have proven a theorem relating their number of components to group theoretical quantities.  

To prove Theorem~\ref{T:FF}, we first prove a result in group theory that we will combine with the result on components from \cite{Ellenberg2012} to give us a count of components of the relevant Hurwitz scheme.  Then, we translate between the surjections counted by $E^{\pm}_{\F_q(t),q^{2n}}(G,G')$ and the actual points on the Hurwitz schemes. Finally, we use comparison to the Hurwitz schemes over $\C$ to uniformly bound the cohomology and carry out the plan described above.  

Note that we do not use any homological stability results, e.g. from \cite{Ellenberg2016}.

\subsection{Group theory computation}\label{S:group}
In this section, we will prove a result in group theory. This result will count $\F_q$-rational components in a moduli space on which we will eventually count points.  However, in this section we isolate the group theory involved.
In this section, we will work with a finite field $\F_q$ (really, just a prime power $q$) and  a finite group $G'$ with $(q,|G'|)=1$.  We use the prime because we will eventually apply these ideas to the group called $G'$ other places in the paper.

First we will define the \emph{universal marked central extension} $\widetilde{G'}$ of a finite group $G'$ for a union $c$ of conjugacy classes of $G'$, following
\cite[Section 7]{Ellenberg2012}.  Let $c$ be a union of conjugacy classes of $G$, whose elements generate $G$, and such that if $g\in c$ and $d$
is an integer relatively prime to the order of $g$, then $g^d\in c$.  Recall the definition of $\widetilde{G'}_c$ from Section~\ref{S:Hdefs}.
The group $\widetilde{G'}$ will also depend (implicitly) on $c$, and be different from $\widetilde{G'}_c$.
We define $\widetilde{G'}$ to be the free group on the formal symbols $[g]$ for $g\in c$, modulo the relations
$[x][y][x]^{-1}=[xyx^{-1}]$ for $x,y\in c$.  See \cite[Section 7]{Ellenberg2012} for  explanation of why this is called a universal marked central extension.  

Consider the group $\widetilde{G'}_c\times_{(G')^{ab}} \Z^{c/G'}$, where $(G')^{ab}$ is the abelianization of $G'$, and $c/G'$ denotes the set of conjugacy classes in $c$, and the map 
 $\Z^{c/G'} \ra (G')^{ab}$ sends
each standard generator to an element of the associated conjugacy class. 
Let $c$ be a union of conjugacy classes $c_i$.  For each $c_i$, we choose one element $g_i\in c_i$ and one lift $\hat{g}_i\in \widetilde{G'}_c$ of $g_i$.  
We then define $\widehat{g g_i g^{-1}}=\tilde{g} \hat{g_i} \tilde{g}^{-1}$ for any choice of lift $\tilde{g}\in\widetilde{G'}_c $ of $g$,
and since $\widetilde{G'}_c\ra G'$ is central, the definition does not depend on our choice of lift of $g$, and since any lift of an element commuting with $g_i$ commutes with $\hat{g_i}$, the definition does not depend on our choice of $g$.

For any set of choices of these $\hat{g}_i$, we have a homomorphism   $\widetilde{G'}\ra\widetilde{G'}_c\times_{(G')^{ab}} \Z^{c/G'}$ taking $[g]$ to $(\hat{g},e_g)$, where $e_g$ is the generator of $\Z^{c/G'}$ corresponding to the conjugacy class of $g$.
In \cite[Theorem 7.5.1]{Ellenberg2012}, it is shown that each of these homomorphisms is actually an isomorphism.
Throughout the paper, we fix one choice for each $\hat{g}_i$ as above, and will use the associated isomorphism $\widetilde{G'}\ra\widetilde{G'}_c\times_{(G')^{ab}} \Z^{c/G'}$.
Even though $\widetilde{G'}_c$ depends on the Schur cover of $G'$ used, from this it follows that the isomorphism type of
$\widetilde{G'}_c\times_{(G')^{ab}} \Z^{c/G'}$ does not depend on this choice.
 We have a map $\widetilde{G'} \ra G'$, given by $[g]\mapsto g$, or by projecting $\widetilde{G'}_c\times_{(G')^{ab}} \Z^{c/G'}$ to its first factor.  

 Recall from Section~\ref{SS:same} that $\hat{\Z}$ is the inverse limit 
$\varprojlim \Z/n\Z$ taken over $n$ relatively prime to $q$. 
We are now going to define an action of $\hat{\Z}^\times$ on the set $\widetilde{G'}$, called the \emph{discrete action} \cite[Section 8.1.7, Equation 9.4.1]{Ellenberg2012}. 
There is an action of $\hat{\Z}^\times$ on any finite group $F$ given by powering.   For an element $g\in F$ and $\{a_n \}_{n} \in \hat{\Z}^\times$ (where $a_n$ is the $\Z/n\Z$ coordinate of the element $\{a_n \}_{n}$), we define $g^{\{a_n \}_n}$ to be $g^{a_d}$, where $d$ is the order of $g$.  This is an action of $\hat{\Z}^\times$ on the set $F$, and it is not necessarily an action by group homomorphisms.  

   Then there is an inherited action of $\hat{\Z}^\times$ on $c$ from the powering action on $G'$, and thus on $c/G'$ by our assumption on $c$ above, and thus an action on $\Z^{c/G'}$.  For $\underline{m}\in \Z^{c/G'}$ and $\alpha\in \hat{\Z}^\times$, we write $\underline{m}^\alpha$ for the result of the action of $\alpha$ on $\underline{m}$.
For $\alpha\in\hat{\Z}^\times$, we define
$$
z_i(\alpha)=\hat{g_i}^{-\alpha} \widehat{g_i^{\alpha}}.
$$
First, we note that $z_i(\alpha)$ is defined by a product in $\widetilde{G'}_c$, but actually lies in 
$H_2(G',c)$ since its image in $G'$ is trivial.  Second, note that if $h=gg_i g^{-1}$ for some $g\in G'$, then
$$
\hat{h}^{-\alpha} \widehat{h^{\alpha}}=(\tilde{g} \hat{g_i} \tilde{g}^{-1} )^{-\alpha}
\widehat{gg_i^\alpha g^{-1}} = \tilde{g} \hat{g_i}^{-\alpha} \tilde{g}^{-1} 
\tilde{g} \widehat{g_i^\alpha} \tilde{g}^{-1}= \tilde{g} \hat{g_i}^{-\alpha}  \widehat{g_i^\alpha} \tilde{g}^{-1}=
 \tilde{g}  z_i(\alpha)\tilde{g}^{-1}=z_i(\alpha).
$$
The last equality follows because $z_i(\alpha)\in H_2(G',c)$, which is in the center of $\widetilde{G'}_c$.
Now we define
$$
Z_{\underline{m}}(\alpha)=\prod_i z_i(\alpha)^{m_i} \in H_2(G',c).
$$
Note that since $z_i(\alpha)\in H_2(G',c)$, the order in which the product is taken does not matter.

   The discrete action of $\hat{\Z}^\times$ on $\widetilde{G'}\isom\widetilde{G'}_c\times_{(G')^{ab}} \Z^{c/G'}$  is defined by
\begin{equation}\label{E:discrete}
\alpha * (h,\underline{m}  )=(h^{\alpha}Z_{\underline{m}}(\alpha),\underline{m}^\alpha).
\end{equation}
To see that this is a group action (on a set), we write 
$(h,\underline{m})=[g_1]^{e_1}\cdots [g_k]^{e_k},$ and thus $h=\hat{g_1}^{e_1}\cdots \hat{g_k}^{e_k}$,
where the $e_i$ are $\pm 1$. 
We then observe
$
\widehat{g_i^\alpha} = \hat{g_i}^{\alpha} z_i(\alpha), 
$
and thus
$$
(\widehat{g_1^\alpha}^{e_1\alpha^{-1}}\cdots \widehat{g_k^\alpha}^{e_k\alpha^{-1}})^\alpha=
((\hat{g_1}^{\alpha} z_i(\alpha))^{e_1\alpha^{-1}}\cdots (\hat{g_k}^{\alpha} z_i(\alpha))^{e_k\alpha^{-1}})^\alpha=h^{\alpha}Z_{\underline{m}}(\alpha),
$$
since the $z_i(\alpha)$ are central.
From the left-hand side it follows that the discrete action defined above is an action of the group $\hat{\Z}^\times$.

The following proposition gives the group theory result that will be necessary in our proof of Theorem~\ref{T:FF}.
\begin{proposition}\label{P:B}
Let $G'$ be a finite  group of even order and $c$ a union of conjugacy classes $c_i$ of order $2$ elements of $G'$.
 Let $q$ be a power of a prime with $(q,|G'|)=1$.
Given $k$, for each $y\in c_k$ and $\underline{n}=(n_1,n_2,\dots)\in \Z^{c/G'}$ with $n_i$ even for $i\ne k$ and $n_k$ odd, there are exactly $|H_2(G',c)[q-1]|$ elements $(x,\underline{n})\in \widetilde{G'}$ such that $x$ has image $y$ in $G'$ and $q^{-1}*(x,\underline{n})=(x,\underline{n})$.
Also, for each $\underline{n}$ with $n_i$ even for all $i$, there are exactly $|H_2(G',c)[q-1]|$ elements $(x,\underline{n})\in \widetilde{G'}$ such that $x$ has image $y=1$ in $G'$ and $q^{-1}*(x,\underline{n})=(x,\underline{n})$.

\end{proposition}

\begin{proof}
First, we note under the given hypotheses on $y$ and $\underline{n}$, that $y$ and $\underline{n}$ have the same image in $(G')^{ab}$ (since all the elements of $c$ have order $2$).  So if $(x,\underline{n})\in \widetilde{G'}_c \times \Z^{c/G}$ with $x$ having image $y$ in $G$, then $(x,\underline{n})\in \widetilde{G'}$.
Note that $q^{-1}*(x,\underline{n})=(x,\underline{n})$ is equivalent to $q*(x,\underline{n})=(x,\underline{n}).$
When $y=1$, we let $\hat{y}=1\in  \widetilde{G'}_c$, and for $y\in c$, we have chosen a lift $\hat{y}$ of $y$ to $\widetilde{G'}_c$ above. 
Since all elements of $c$ have order $2$ and $q$ is odd, we have
$$
q*(x,\underline{n})=(x^q \prod_i (\hat{g_i}^{-q} \widehat{g_i^{q}})^{n_i},\underline{n})=(x^q \prod_i (\hat{g_i}^{-q+1})^{n_i},\underline{n}).
$$
 
The preimages of $y$ in $ \widetilde{G'}_c$ are
the elements $\hat{y}h$ for $h\in H_2(G',c)$.  
Thus, we ask for how many elements $h\in H_2(G',c)$ is
\begin{equation}\label{E:ask}
(\hat{y}h)^q \prod_i (\hat{g_i}^{-q+1})^{n_i} =\hat{y}h?
\end{equation}
Since $h$ is central, the above is equivalent to
$$
\hat{y}^{q-1} \prod_i (\hat{g_i}^{-q+1})^{n_i} =h^{1-q}.
$$
 Note that if $y\neq 1$, we have $\hat{y}^{q-1}=\hat{g_k}^{q-1}$,
since if $y=gg_kg^{-1}$, we have
$$\hat{y}^{q-1} =(\tilde{g} \hat{g_k} \tilde{g}^{-1} )^{q-1}=\tilde{g} \hat{g_k}^{q-1} \tilde{g}^{-1}=\hat{g_k}^{q-1} ,
$$ where the last equality is since $\hat{g_k}^{q-1} $ has trivial image $G'$ and thus is central.
For $y\ne 1$, we have 
$$\hat{y}^{q-1} \prod_i (\hat{g_i}^{-q+1})^{n_i} = (\hat{g_k}^{-q+1})^{n_k-1}\prod_{i\ne k} (\hat{g_i}^{-q+1})^{n_i}$$
is in $H_2(G',c)^{q-1}$, since $\hat{g_i}^2\in H_2(G',c)$ for all $i$ and by the parity of the $n_i$.
So in Equation~\eqref{E:ask}, the number of such $h$ is $|H_2(G,c)[q-1]|,$ which proves the lemma for $y\ne 1$.
When $y=1$, 
$$
\prod_i (\hat{g_i}^{-q+1})^{n_i}$$
is similarly in $H_2(G,c)^{q-1}$, and the last case of the lemma follows similarly.
\end{proof}

\begin{remark}
From the proof of Proposition~\ref{P:B}, one can see that depending on the image of $y$ and the elements of $c$ in $(G')^{ab}$, as well as the image of $\hat{g}^2\in H_2(G',c)$ for $g\in c$, we can define a function on the $n_i$'s modulo $2$ that determines whether
there are any elements $(x,\underline{n})\in \widetilde{G'}$ such that $x$ has image $y$ in $G'$ and $q^{-1}*(x,\underline{n})=(x,\underline{n})$, and whenever there are such elements $(x,\underline{n})$ there are exactly $|H_2(G',c)[q-1]|$ of them.

\end{remark}

\subsection{Translation from surjections to extensions of the base field}
In this section, we translate between the surjections counted by $E^{\pm}_{\F_q(t),q^{2n}}(G,G')$ and the \emph{marked $(G',c)$ extensions} that the Hurwitz schemes will parametrize.

\begin{definition}
Let $Q$ be a global field with a  place ${\v}$.  Let $G'$ be a finite group, and $c$ a union of conjugacy classes of $G'$. We fix a separable closure $\bar{Q}_{\v}$ of the completion $Q_{\v}$.  Then, inside $\bar{Q}_{\v}$ we have the separable closure $\bar{Q}$ of $Q$.  This gives a map $\Gal(\bar{Q}_{\v}/Q_{\v})  \ra \Gal(\bar{Q}/Q)$, and in particular  distinguished decomposition and  inertia groups  in $\Gal(\bar{Q}/Q)$ at ${\v}$.  
We define (as in \cite[Section 10.2]{Ellenberg2012}) a \emph{marked $(G',c)$ extension} of $Q$ to be $(M,\pi,m)$
such that 
$M/Q$ is a Galois extension of fields, 
$\pi$ is an isomorphism $\pi: \Gal(M/Q)\isom G'$ such that all inertia groups in $\Gal(M/Q)$ (except for possibly the one at ${\v}$) have image in $\{1\}\cup c$, and $m$, the \emph{marking}, is a homomorphism $M_{\v}:=M\tensor_Q Q_{\v} \ra \bar{Q}_{\v}$. 
 Note that restriction to $M$ gives a bijection between homomorphisms 
$M_{\v} \ra \bar{Q}_{\v}$ and homomorphisms $M\ra \bar{Q}$. 
Two marked $(G',c)$ extensions $(M_1,\pi_1,m_1)$ and $(M_2,\pi_2,m_2)$ are isomorphic when there is an isomorphism $M_1\ra M_2$ taking $\pi_1$ to $\pi_2$ and $m_1$ to $m_2$.  
The marking $m$ in a marked $(G',c)$ extension $(M,\pi,m)$ gives a map $\Gal(\bar{Q}_{\v}/Q_{\v}) \ra \Gal(M/Q)$.
 Composing with $\pi$ we get an \emph{infinity type} $\Gal(\bar{Q}_{\v}/Q_{\v}) \ra G'$. 
\end{definition}

 The translation is done by the following result, which is similar to \cite[Theorem 5.1]{Boston2017}.
\begin{proposition}\label{P:trans}
Let $Q$ be a global field with a choice of place $\v$. Let $G$ be a finite group and $G'$ a subgroup of
$G \wr S_2$.  Let $c$ be the order $2$ elements of $G'$ with non-trivial image in $S_2$.

Let $\phi: \Gal(\bar{Q}_\v/Q_\v) \ra G'$ be a ramified homomorphism with image $\langle (1,\sigma) \rangle$.
There is a bijection between
\begin{align*}
&\{(K,\rho) | \\
&K\sub \bar{Q}, [K:Q]=2, K_{\v}/Q_{\v} \textrm{ the quadratic extension given by $\ker(\phi)$, } \rho\in \Sur(G_K^{\un,\v}, G) | \rho \textrm{ is type }G' \}
\end{align*}
and
$$
\{\textrm{isomorphism classes of marked $(G',c)$ extensions $(L,\pi,m)$ of $Q$ with infinity type $\phi$}  \}.
$$

There is also a bijection between
\begin{align*}
\{(K,\rho,y) |& K\sub \bar{Q}, [K:Q]=2, K_{\v}/Q_{\v} \textrm{ split completely, } \rho\in \Sur(G_K^{\un,\v},G), 
\\
& y \textrm{ a twist for $\rho$}, (\rho,y) \textrm{ of type $G'$} \}
\end{align*}
and
$$
\{\textrm{isomorphism classes of marked $(G',c)$ extensions $(M,\pi,m)$ of $Q$ with trivial infinity type}  \}.
$$ 
In these bijections, we have $\Disc(M)=\Disc(K)^{|G'|/2}$.
\end{proposition}  

\begin{proof}
Note that in each isomorphism class of  marked $(G',c)$ extensions of $Q$, there is a distinguished element such that $M\sub \bar{Q}$ and $m|_{M}$ is the inclusion map.  

In either case, we have $L\sub \bar{Q}$ corresponding to $\ker(\rho)$ and $\tilde{L}$ the Galois closure of $L$ over $Q$.
By the definition of type, we have $\pi: \Gal(\tilde{L}/ Q)\isom G'$.
Then $\tilde{L}$ is a marked $G'$ extension of $Q$, where $\tilde{L}\sub \bar{Q}$ and the marking is inclusion.  

Given a marked $(G',c)$ extension $(M,\pi,m)$ of $Q$ such that $M\sub \bar{Q}$, and $m$ is the inclusion, we obtain $K$ as the fixed field of the kernel of the map to $S_2$.  The isomorphism $\pi: \Gal(M/Q)\ra G'$ restricts to a homomorphism
$\pi|_{\Gal(M/K)}: \Gal(M/K) \ra G\times G$, and we obtain $\rho$ by projection onto the first factor.  In the second case, we have $y=\pi^{-1}(1,1,\sigma)$.  
It is straightforward to check that these constructions are inverses to each other, and respect the other properties in the theorem.
\end{proof}

\subsection{Properties of the Hurwitz scheme}

In the following theorem, we recall the properties of the Hurwitz scheme constructed by Ellenberg, Venkatesh, and Westerland, building on work by Romagny and Wewers.
Recall the definition of $\widetilde{G'}$ from Section~\ref{S:group} and $\widetilde{G'}\langle -1 \rangle$ from Section~\ref{SS:same}, where we use the discrete action of $\hat{\Z}^\times$ on $\widetilde{G'}$ to define $\widetilde{G'}\langle -1 \rangle$.  We have an automorphism of $\Z(1)$ given by taking $q$th powers, and for $f\in\widetilde{G'}\langle -1 \rangle$, we define $f^F(u)=f(u^{q^{-1}})=q^{-1} * f(u)$.

\begin{theorem}[Ellenberg, Venkatesh, and Westerland]\label{T:Chur}\footnote{The paper \cite{Ellenberg2012} has been temporarily withdrawn by the authors because of a gap which affects
Sections 6, 12 and some theorems of the introduction of \cite{Ellenberg2012}. That gap does not affect any of the
results from \cite{Ellenberg2012} that we use in this paper.}
Let $G'$ be a finite group with trivial center and let $c$ be a union of conjugacy classes of order $2$ elements of $G'$, and such that the elements of $c$ generate $G'$. We let $c/G'$ denote the set  $\{c_1,c_2,\dots\}$ of conjugacy classes of $c$.
 Let $\F_q$ be a finite field with $q$ relatively prime to $| G'|$.
 Let $\underline{n}=(n_1,n_2,\dots)\in \Z^{c/G}$ with $n=\sum_i n_i$.
There is a Hurwitz scheme $\CCHur_{G',n}$ over $\Z[| G'|^{-1}]$ constructed in \cite[Section 8.6.2]{Ellenberg2012} with the following properties:
 \begin{enumerate}

\item We have $\CCHur_{ G',n}$ is a finite \'etale cover of the relatively smooth $n$-dimensional configuration space $\Conf^n$ of $n$ distinct unlabeled points in $\A^1$ over $\Spec \Z[| G'|^{-1}]$.
\label{i:smooth} 
 
 \item  \label{i:bij} 
 For $g\in c\cup 1$, the scheme $\CCHur_{ G',n}\tensor \F_q$ has an open and closed subscheme $\CCHur^{c,g}_{ G',\underline{n}}$
 (which is over $\F_q$, but the $q$ is suppressed from the notation) such that there is a bijection between
 \begin{enumerate}
 \item isomorphism classes of  marked $(G',c)$-extensions $M$ of $\F_q(t)$ such that for $i\geq 1$ the total degree of the non-infinite places of $\F_q(t)$ with inertia in $c_i$ is $n_i$
 and an infinity type $\phi$ such that $\phi(F_\Delta)=1$ and $\phi$ sends a generator of tame inertia to $g$ (where $F_\Delta$ is a lift of Frobenius to $\Gal(\overline{\F_q(t)_{\v}}/\F_q(t)_{\v})$ that acts trivially on $\F_q((t^{-1/\infty}))$). 
 \item points of $\CCHur^{c,g}_{ G',\underline{n}}(\F_q)$ \cite[Section 10.4]{Ellenberg2012}. 
 \end{enumerate}

\item 
Let $\Frob$ denote the Frobenius automorphism on $\CCHur^{c,g}_{G',\underline{n}}  \tensor {\bar{\F}_q}$, induced by the automorphism of $\bar{\F}_q$ over $\F_q$ sending $a\mapsto a^q$. 
Given $G'$, there is an $N_0$, only depending on $G'$ such that for all $\underline{n}$ with $n_i\geq N_0$ for all $i$, the components of $\CCHur^{c,g}_{G',\underline{n}}  \tensor {\bar{\F}_q}$ fixed by $\Frob$ are in bijection with elements  $f\in \widetilde{G'}\langle -1 \rangle$ such that $f^F=f$ and for any $u\in \hat{\Z}(1)^\times$ we have $f(u)=(x,\underline{n})$, where $x$ has image $g$ in $G'$  \cite[Theorem 8.7.3]{Ellenberg2012} (see Section~\ref{S:group} for definitions).  For a component of $\CCHur^{c,g}_{ G',\underline{n}}$ in bijection with $f\in \widetilde{G'}\langle -1 \rangle$, an $\F_q$ point of that component corresponds to an extension via \eqref{i:bij} above, and
the component invariant of that extension (defined in Section~\ref{SS:same}) is $f$.

\label{i:comp}

 \end{enumerate}
\end{theorem} 

\begin{remark}
The scheme $\CCHur^{c,g}_{G',\underline{n}}$ comes from restricting to the parametrization of covers of $\P^1$  all of whose local inertia groups have image 
 in $c \cup \{1\}$ and whose inertia at infinity has image $g$ in $G'$.  The argument that $\CCHur^{c,g}_{G',\underline{n}}$ is an open and closed subscheme is as in \cite[Section 7.3]{Ellenberg2016}. 
\end{remark}

 \subsection{Counting $\F_q$ points}
In this section, we will count points of $\CCHur_{G',n}$ in Theorem~\ref{T:count}, which, when combined with Propositions~\ref{P:B} and \ref{P:trans} will allow us to finally prove Theorem~\ref{T:FF}.

\begin{theorem}\label{T:count}
Let $G'$ and $c$ be as in Theorem~\ref{T:Chur}.  
Let $g\in c\cup \{1\}$.
For each positive integer $n$ there is a constant $K_n$ such that for
 $(q,|G'|)=1$ and $\underline{n}$ with $\sum_i n_i=n$, and 
 $C$  a  $\Frob$ fixed component of $\CCHur^{c,g}_{G',\underline{n}}  \tensor {\bar{\F}_q}$,
 we have
$$
|\#C({\F}_q) - q^{n}|\leq  K_n q^{n-1/2}.
$$
\end{theorem}
\begin{proof}
Our theorem will follow by applying the Grothendieck-Lefschetz trace formula to $C$.
 By Theorem~\ref{T:Chur} \eqref{i:smooth}, we have that $X:= \CCHur^{c,g}_{G',\underline{n}} \tensor {{\F}_q}$ is smooth of dimension $n$.
 We have that $\dim H^i_{\textrm{c,\'et}}(X_{\bar{\F}_q},\Q_\ell)=
 \dim H^{2n-i}_{\textrm{\'et}}(X_{\bar{\F}_q},\Q_\ell)$ by Poincar\'{e} Duality.
 
Next, we will relate $\dim H^{j}_{\textrm{\'et}}(X_{\bar{\F}_q},\Q_\ell)$ to 
 $\dim H^{j}( \CCHur^{c,g}_{G',\underline{n}}(\C),\Q_\ell)$ for some $\ell>n$.  To compare \'{e}tale cohomology between characteristic $0$ and positive characteristic, we will use \cite[Proposition 7.7]{Ellenberg2016}.  The result \cite[Proposition 7.7]{Ellenberg2016} gives an isomorphism between \'{e}tale cohomology between characteristic $0$ and positive characteristic in the case of a finite cover of a complement of a reduced normal crossing divisor in a smooth proper scheme.  Though \cite[Proposition 7.7]{Ellenberg2016} is only stated	 for \'{e}tale cohomology with coefficients in $\Z/\ell\Z$, the argument goes through identically for coefficients in $\Z/\ell^k\Z$, and then we can take the inverse limit and tensor with $\Q_{\ell}$ to obtain the result of \cite[Proposition 7.7]{Ellenberg2016} with $\Z/\ell\Z$ coefficients replaced by $\Q_\ell$ coefficients.  So we apply this strengthened version to conclude that 
$ \dim H^{j}_{\textrm{\'et}}(X_{\bar{\F}_q},\Q_\ell)=\dim H^{j}_{\textrm{\'et}}(( \CCHur^{c,g}_{G',\underline{n}})_{\C},\Q_\ell)$.
(As in \cite[Proof of Proposition 7.8]{Ellenberg2016}, we apply comparison to $ \CCHur^{c,g}_{G',\underline{n}} \times_{\Conf^n } \operatorname{PConf}_n,$ where $\operatorname{PConf}_n$ is the moduli space of $n$ labelled points on $\A^1$, and is the complement of a relative normal crossings divisor in a smooth proper scheme \cite[Lemma 7.6]{Ellenberg2016}.
Then we take $S_n$ invariants to compare the \'etale cohomology of $\CCHur^{c,g}_{G',\underline{n}}$ across characteristics.)
By the comparison of \'{e}tale and analytic cohomology \cite[Expos\'e XI, Theorem 4.4]{1973} $\dim H^{j}( \CCHur^{c,g}_{G',\underline{n}}(\C),\Q_\ell)=\dim H^{j}_{\textrm{\'et}}(( \CCHur^{c,g}_{G',\underline{n}})_{\C},\Q_\ell)$.  
Given $n,$ let $B(n)=\max_{k,\underline{n},1\leq j< 2n} \dim H^{j}( \CCHur^{c,g}_{G',\underline{n}}(\C),\Q_\ell),$
where the max is over $\underline{n}$ with $n=\sum_{i} n_i$.

Then by the Grothendieck-Lefschetz trace formula we have
$$
\#C({\F}_q) =\sum_{j\geq 0} (-1)^j\Tr(\Frob|_{ H^j_{\textrm{c,\'et}}(C,\Q_\ell)})
$$
and 
also we know $\Tr(\Frob|_{ H^{2n}_{\textrm{c,\'et}}(C,\Q_\ell)})$ is $q^{n}$, since $C$ is a component of $X_{\bar{\F}_q}$ fixed by $\Frob$.
Since $X$ is smooth, we have that the absolute value of any eigenvalue of $\Frob$ on $H^j_{\textrm{c,\'et}}(X_{\bar{\F}_q},\Q_\ell)$ is at most
$q^{j/2}$ (from Deligne's theory of weights on \'etale cohomology), and thus similarly for eigenvalues of $\Frob$ on $H^j_{\textrm{c,\'et}}(C,\Q_\ell)$.
Thus, 
\begin{align*}
\left|\#C({\F}_q)  - q^{n} \right |&=\left| \sum_{0\leq j< 2n} (-1)^j\Tr(\Frob|_{ H^j_{\textrm{c,\'et}}(C,\Q_\ell)}) \right| \\
&\leq \sum_{0\leq j< 2n} q^{j/2} B(n) .
\end{align*}
 The theorem follows, with $K_n=\sum_{0\leq j< 2n} B(n)$.
\end{proof}

Here we will give a stronger version of Theorem~\ref{T:FF}. 
Recall from Section~\ref{S:Edefs} that a for a quadratic extension $K/\F_q(t)$ and a $\rho\in \Sur(G_K^{\un,\v}, G)$ of type $G'$, we have constructed a map $G_{\F_q(t)}\ra G'$ (also using a twist for $\rho$ in the real case).
 Recall for a map $G_{\F_q(t)}\ra G'$ with inertia in $c$ corresponding to an extension of discriminant with norm $q^{2n}$, we have defined in Section~\ref{SS:same} a component invariant  $f\in\widetilde{G'}\langle -1 \rangle$ such that for any $u\in\hat{\Z}(1)^\times$ we have $f(u)=(x,\underline{n})\in \tilde{G'_c}\times_{(G')^{ab}} \Z^{c/G'}$ with $\sum_i n_i=2n-\epsilon$ (where $\epsilon$ is $1$ if $K$ is ramified at infinity and $0$ otherwise).

\begin{theorem}\label{T:SFF}
Consider a finite group $G$ and an \sa subgroup $G'\sub G\wr S_2$ (where $G'$ has trivial center),
and $c$ the set of order $2$ elements of $G' \setminus \ker(G'\ra S_2)$. 
For each prime power $q$ with $(q,|G'|)=1$, let  $f_q \in\widetilde{G'}\langle -1 \rangle$ with $f_q^F=f_q$
and  $u_q\in\hat{\Z}(1)^\times$.
 If $G'$ is good, for $n$ sufficiently large, 
 if $f_q(u_q)=(x,2n-1)$ and $x$ has image $(1,1,\sigma)$ in $G'$ then
   $$
\lim_{\substack{q\ra\infty\\(q,|G'|)=1}}\frac{\sum_{K\in IQ_{=q^{2n}}}
   \#\{\rho\in \Sur(G_K^{\un,\v}, G) | \rho \textrm{ is type }G',\textrm{component invariant $f_q$}\}    }{\sum_{K\in IQ_{=q^{2n}}} 1} =1,
    $$ 
 and  if $f_q(u_q)=(x,2n)$ and $x$ has image $1$ in $G$ then
 $$
\lim_{\substack{q\ra\infty\\(q,|G'|)=1}} \frac{\sum_{K\in RQ_{=q^{2n}}}  \#\{(\rho,y) | \rho\in \Sur(G_K^{\un,\v}, G), y\textrm{ \twist for }\rho, (\rho,y) \textrm{  type }G', \textrm{component invt. $f_q$}\}      }{\sum_{K\in RQ_{=q^{2n}}} 1}=1.$$
If $G'$ is not good, then there is a constant $w_G$ such that for $n$ sufficiently large we have
\begin{align*}
\lim_{\substack{q\ra\infty\\(q,|G'|)=1}}
\frac{\tilde{E}^{\pm}_{\F_q(t),q^{2n}}(G,G')}{|H_2(G',c)[|\mu_{\F_q(t)}|]|}
\geq w_G n^{N_{G'}-1}, 
\end{align*}
where $\mu_{\F_q(t)}$ is the group of roots of unity of $\F_q(t)$, so $|\mu_{\F_q(t)}|=q-1$, and $N_{G'}$ is the number of conjugacy classes of order $2$ elements of $G'$ not in $\ker \pi$.
\end{theorem}

\begin{remark}
Note that in the first two limits above, $f_q$ changes with $q$.  However, since the component invariants are in 
$\widetilde{G'}\langle -1 \rangle$, which implicitly depends on $q$, there is no way to talk about a component invariant that is independent of $q$.  (Even if there were  a way to identify component invariants between different $q$, the statement above is stronger than a statement that requires a relationship between $f_{q}$ and $f_{q'}$.)
\end{remark}

Next we confirm that Theorem~\ref{T:SFF} does indeed imply the first statement of Theorem~\ref{T:FF}.  (The second statements of both theorems are the same.)
\begin{proof}[Proof of Theorem~\ref{T:FF}]
We  will sum Theorem~\ref{T:SFF} over all possible component invariants.  
The image of $f_q(u_q)$ in $G'$ is a generator for inertia at $\infty$ (from the definition of the component invariant).
The second component of $f_q(u_q)$ is the norm of the discriminant of $K$ without the contribution of 
the ramification at $\infty$ (from the definition of the component invariant).
 We have $f_q^F=f_q$ if and only if $f_q(u_q)=(x,m)$ with $q^{-1}
*  (x,m)=(x,m)$  (by definition of $f_q^F$).  
In the imaginary quadratic case, by construction of the map $G_{\F_q(t)}\ra G'$ associated $\rho$, 
the inertia at $\infty$ maps to $(1,1,\sigma)$.  
In the real quadratic case, in the map $G_{\F_q(t)}\ra G'$,
the inertia at $\infty$ maps to $1$. 
Thus the component invariants handled by Theorem~\ref{T:SFF} are all possible component invariants in the imaginary and real quadratic cases.  
By Theorem~\ref{T:Chur} \eqref{i:comp} and Proposition~\ref{P:B}, there are $|H_2(G',c)[q-1]|$ of these component invariants in either case.  If we restrict to $q$ such that $|H_2(G',c)[q-1]|$ is fixed, then we have
$$
\lim_{\substack{q\ra\infty\\(q,|G'|)=1\\ |H_2(G',c)[q-1]|=d }} \tilde{E}^\pm_{\F_q(t),q^n}(G,G') =d,
$$
by summing Theorem~\ref{T:SFF} over all possible component invariants.  Then Theorem~\ref{T:FF} follows by dividing both sides by $d$, and the translation between $\tilde{E}$ and $E$.
\end{proof}

\begin{remark}
When, as in Theorem ~\ref{T:SFF}, we take a limit as $q\ra\infty$, either before or after a limit in $n$, we lose constants that go to $1$ as $q\ra\infty$, like zeta values.  When $G$ is abelian, we already expected there are no such factors in the moments from the Cohen-Lenstra heuristics.  Additionally, when $G$ is abelian,  the Cohen-Lenstra-Martinet heuristics predict that the moment does not  depend on $q$  when $(q-1,|G'|)$ is fixed (as that fixes the relevant roots of unity in the base field $\F_q(t)$). 
\end{remark}

\begin{proof}[Proof of Theorem~\ref{T:SFF}]

We let $Q=\F_q(t)$ and $Q_\infty=\F_q((t^{-1}))$, We have that $q$ is odd since $|G'|$ is even. 
Let $c$ be the union of conjugacy classes $c_i$ of $G'$ of order two elements with non trivial image in $S_2$ and  let $(1,1,\sigma)\in G'$ be in $c_1$.

We first prove the first statement in the imaginary quadratic case. 
We wish to understand the average
  $$
\frac{\sum_{K\in IQ_{=q^{2n}}}
   \#\{\rho\in \Sur(G_K^{\un,\v}, G) | \rho \textrm{ is type }G',\textrm{ component invariant $f_q$}\}    }{\sum_{K\in IQ_{=q^{2n}}} 1}.
    $$
Let $\phi$ be the  infinity type  sending $F_\Delta$ to $1$ and inertia to $(1,1,\sigma)$.  
By Theorem~\ref{T:Chur}~\eqref{i:bij}, we have a bijection between 
$\CCHur^{c,(1,1,\sigma)}_{G',\underline{n}}(\F_q)$ and 
isomorphism classes of  marked $(G',c)$-extensions $M/\F_q(t)$ of infinity type $\phi$
with the total degree of  non-infinite places of $\F_q(t)$ with inertia in $c_i$ is $n_i$ for all $i\geq 1$.
Note that infinity type $\phi$ is exactly what makes the associated quadratic extension imaginary quadratic (this was essentially the definition of imaginary quadratic).
In particular these extensions have  $\Nm \Disc (M)=q^{( \sum_i n_i+1)|G'|/2},$ and thus associated quadratic extension with discriminant norm $q^{ \sum_i n_i+1},$
Moreover, the $\F_q$ points in the component that is associated to $f_q$ under Theorem~\ref{T:Chur} \eqref{i:comp} correspond exactly to those $\rho$ with $G_{\F_q(t)}\ra G'$ with component invariant $f_q$.
By Proposition~\ref{P:trans} and Theorem~\ref{T:count}, we then conclude that if $c$ is a single conjugacy class that
\begin{align*}
\#\{(K,\rho) | K\in IQ_{=q^{2n}}, \rho\in \Sur(G_K^{\un,\v}, G) | \rho \textrm{ is type }G'  ,\textrm{ component invariant $f_q$}\}
 =& q^{2n-1}+O_n(q^{2n-3/2}).
\end{align*}
Here the $O_n$ notation indicates that the implied constant may depend on $n$.
An imaginary quadratic extension in $IQ_{=q^{2n}}$ can be given uniquely as $\F_q(t)(\sqrt{f(t)})$ for a monic, square-free degree $2n-1$ polynomial $f\in\F_q[t]$, and thus  $\#IQ_{=q^{2n}}=q^{2n-1}-q^{2n-2}$ since it is classical that there are $q^{2n-1}-q^{2n-2}$ such $f(t)$. 
Thus in this case, when $n$ is sufficiently large, we have
\begin{align*}
&\frac{\sum_{K\in IQ_{=q^{2n}}} \# \{\rho | \rho\in \Sur(G_K^{\un,\v}, G) | \rho \textrm{ is type }G' ,\textrm{ component invariant $f_q$}\}}
{\#IQ_{=q^{2n}}
}\\
=&\frac{ q^{2n-1} +O_n(q^{2n-3/2})}{q^{2n-1}-q^{2n-2}}=1+O_n(q^{-1/2}).
\end{align*}
The  $G'$ good part of the theorem follows for imaginary quadratic extensions.

Now consider the case when the  set of conjugacy classes $c/G'$ in $c$ has more than $1$ element.
We have
\begin{align*}
\#\{(K,\rho) | K\in IQ_{=q^{2n}}, \rho\in \Sur(G_K^{\un,\v}, G) | \rho \textrm{ is type }G'  \}
=  \sum_{\substack{\underline{n} \\ \sum_i n_i =2n-1}} {\#\CCHur^{c,(1,1,\sigma)}_{G,\underline{n}}(\F_q)}.
\end{align*}
There is a  constant $v_{G'}$ such that for $n$ sufficiently large, there are at least $v_{G'} n^{\# c/G'-1}$ choices  $\underline{n}$ with $\sum_i n_i=2n-1$ such that all the $n_i$ are large enough for the application of $\ref{T:Chur} \eqref{i:comp}$, and such that  the $n_i$ are even for $i>1$ and $n_1$ is odd.  Let  $\mathcal{C}$ denote the set of such choices of $\underline{n}$. 
So for  $\underline{n}\in \mathcal{C}$, we have the number of $\Frob$ fixed components of $\CCHur^{c,(1,1,\sigma)}_{G,\underline{n}}$
is $|H_2(G',c)[q-1]|$ by Theorem~\ref{T:Chur} \eqref{i:comp} and Proposition~\ref{P:B}.
 Then we apply Theorem~\ref{T:count}, and we have
\begin{align*}
 \sum_{\substack{\underline{n} \\ \sum_i n_i =2n-1}} {\#\CCHur^{c,(1,1,\sigma)}_{G,\underline{n}}(\F_q)}& \geq
\sum_{\underline{n} \in \mathcal{C}} {|H_2(G,c)[q-1]|q^n+O_n(q^{n-1/2})}\\
& \geq
v_{G'} n^{\# c/G'-1}|H_2(G',c)[q-1]|q^n+O_n(q^{n-1/2}).
\end{align*}
Thus 
\begin{align*}
\tilde{E}^{-}_{\F_q(t),q^{2n}}(G,G')=&
\frac{\sum_{K\in IQ_{=q^{2n}}} \# \{\rho | \rho\in \Sur(G_K^{\un,\v}, G) | \rho \textrm{ is type }G'\}}
{\#IQ_{=q^{2n}}
}
\\
\geq &\frac{v_{G'} n^{\# c/G'-1}|H_2(G',c)[q-1]|q^{2n-1}+O_n(q^{2n-3/2})}{q^{2n-1}-q^{2n-2}}.
\end{align*}
Dividing both sides by $|H_2(G',c)[q-1]|$, and then letting $q\ra\infty$, the second part of the theorem follows for imaginary quadratic extensions.
 
 Now we consider the real quadratic case. We are interested in the average
 $$
\frac{\sum_{K\in RQ_{=q^{2n}}}  \#\{(\rho,y) | \rho\in \Sur(G_K^{\un,\v}, G), y\textrm{ \twist for }\rho, (\rho,y) \textrm{ is type }G' ,\textrm{ component invariant $f_q$}\}      }{\sum_{K\in RQ_{=q^{2n}}} 1}.
    $$
By Theorem~\ref{T:Chur}~\eqref{i:bij}, we have a bijection between 
$\CCHur^{c,1}_{G',\underline{n}}(\F_q)$ and 
isomorphism classes of  marked $(G',c)$-extensions $M/\F_q(t)$ of trivial infinity type
with the total degree of  non-infinite places of $\F_q(t)$ with inertia in $c_i$ is $n_i$ for all $i\geq 1$.
Note that trivial infinity type $\phi$ is exactly what makes the associated quadratic extension real quadratic (this was  the definition of real quadratic).
In particular these extensions have  $\Nm \Disc (M)=q^{( \sum_i n_i+1)|G'|/2},$ and thus associated quadratic extension with discriminant norm $q^{ \sum_i n_i+1},$
Moreover, the $\F_q$ points in the component that is associated to $f_q$ under Theorem~\ref{T:Chur} \eqref{i:comp} correspond exactly to those $\rho$ with $G_{\F_q(t)}\ra G'$ with component invariant $f_q$.
By Proposition~\ref{P:trans} and Theorem~\ref{T:count}, we then conclude that if $c$ is a single conjugacy class that
\begin{align*}
&\sum_{K\in RQ_{=q^{2n}}}  \#\{(\rho,y) | \rho\in \Sur(G_K^{\un,\v}, G), y\textrm{ \twist for }\rho, (\rho,y) \textrm{ is type }G' ,\textrm{ component invariant $f_q$}\}   \\ 
 =& q^{2n}+O_n(q^{2n-1/2}).
\end{align*}
Here the $O_n$ notation indicates that the implied constant may depend on $n$.
A real quadratic extension in $RQ_{=q^{2n}}$ can be given uniquely as $\F_q(t)(\sqrt{f(t)})$ for a monic, square-free degree $2n$ polynomial $f\in\F_q[t]$, and thus  $\#RQ_{=q^{2n}}=q^{2n}-q^{2n-1}$ since it is classical that there are $q^{2n}-q^{2n-1}$ such $f(t)$. 
Thus in this case, when $n$ is sufficiently large, we have
\begin{align*}
&\frac{\sum_{K\in RQ_{=q^{2n}}}  \#\{(\rho,y) | \rho\in \Sur(G_K^{\un,\v}, G), y\textrm{ \twist for }\rho, (\rho,y) \textrm{ is type }G' ,\textrm{ component invariant $f_q$}\}      }{\sum_{K\in RQ_{=q^{2n}}} 1}\\
=&\frac{ q^{2n} +O_n(q^{2n-1/2})}{q^{2n}-q^{2n-1}}=1+O_n(q^{-1/2}).
\end{align*}
The  $G'$ good part of the theorem follows for real quadratic extensions.

If $c$ is more than one conjugacy class, the $G'$ not good real quadratic case follows with an analogous argument to that in the imaginary quadratic case.
\end{proof}

\section{A refined conjecture}\label{S:Refine}
In light of the lifting invariant defined in Section~\ref{S:lift} and the statement of Theorem~\ref{T:SFF}, it is natural to conjecture that the average number of unramified $G$ extensions of quadratic fields of type $G'$ is $1$ per lifting invariant (counting in the rigidified $\tilde{E}$ sense).  However, we have only defined the lifting invariant for tame extensions of global fields.  Therefore, at this point we can only make a well-posed conjecture if we restrict to tamely ramified quadratic fields.  This is a local condition at $2$ on the quadratic fields.  It is predicted that such local conditions do not affect the Cohen-Lenstra heuristics and in particular the conjectured averages in the case that $G$ is abelian (see \cite{Bhargava2015d,Bhargava2016,Wood2017b}).  

Recall from Section~\ref{S:Edefs} that for a quadratic extension $K/Q$, a $\rho\in \Sur(G_K^{\un,\v}, G)$ (with a twist $y$ in the real quadratic case) gives a map $\phi:G_Q\ra G'$, and in Section~\ref{S:defI} we have defined a lifting invariant $I_{G',\pi,Q}(\phi,u)$, where $\pi$ is the map $G'\ra S_2$ and $u\in \mu_{Q(\mu_{4|\widetilde{G'}_c|})}$.  By an abuse of notation, we write 
$I_{G',\pi,Q}(\rho,u):=I_{G',\pi,Q}(\phi,u)$ in the imaginary case and $I_{G',\pi,Q}((\rho,y),u):=I_{G',\pi,Q}(\phi,u)$ in the real case (we can technically tell which form of the notation we mean by checking the first subscript of the $I$ against the range of the first argument).  
We write $IQ_X^t$ and $RQ_X^t$ for the imaginary and real quadratic tame extensions of $Q$, respectively, with discriminant norm at most $X$.
We can now state precisely a version of Conjecture~\ref{C:E} stratified by lifting invariant.

\begin{conjecture}
Let $Q$ be $\Q$ or $\F_q(t)$.  
Consider a finite group $G$ and an \sa subgroup $G'\sub G\wr S_2$ (with surjection $\pi:G' \ra S_2$),
and $c$ the set of order $2$ elements of $G' \setminus \ker\pi$. 
Let $h\in H_2(G',c){|\mu_{Q}|}$.  Let $u\in \mu_{Q(\mu_{4|\widetilde{G'}_c|})}$.
 If $G'$ is good,  then
   $$
\lim_{X\ra\infty} \frac{\sum_{K\in IQ^t_{X}}
   \#\{\rho\in \Sur(G_K^{\un,\v}, G) | \rho \textrm{ is type }G', I_{G',\pi,Q}(\rho,u)=h\}    }{\sum_{K\in IQ^t_{X}} 1} =1,
    $$ 
 and  
 $$
\lim_{X\ra\infty} \frac{\sum_{K\in RQ^t_{X}}  \#\{(\rho,y) | \rho\in \Sur(G_K^{\un,\v}, G), y\textrm{ \twist for }\rho, (\rho,y) \textrm{  type }G', I_{G',\pi,Q}((\rho,y),u)=h\}      }{\sum_{K\in RQ^t_{X}} 1}=1.$$
\end{conjecture}

It would be interesting to have a definition of the lifting invariant that extended to wild extensions so that one could make a conjecture without a restriction to tame quadratic fields.

\section{The Malle-Bhargava principle motivation for Conjecture~\ref{C:E}}\label{S:MB}

In this section, we will apply the Malle-Bhargava principle to suggest for us, in Conjecture~\ref{C:E},  when $E^{\pm}(G,G')$ should be finite versus infinite.

Malle \cite{Malle2004}, based on the local possibilities for the discriminant of a $G$-number field, gave a conjecture for the (global) asymptotics of $G$-fields by discriminant.  (This conjecture has known counterexamples--see below.)
 Bhargava \cite{Bhargava2010} gave a different heuristic, also based on local possibilities for the discriminant, 
 to predict both the count of $G$-fields (in which case the predicted asymptotic agrees with Malle's conjecture) and the count of $G$-fields with local conditions at finitely many primes.
 (Bhargava's principle was stated as a question, since there were known cases in which it fails.)
   We apply this principle to count number fields with given local conditions at infinitely many primes.  While the  version we apply is not discussed by Malle or Bhargava, it gives provably correct predictions in some cases that are closely related to our situation including the Cohen-Lenstra-Martinet applications to the $3$-part of the class group of quadratic fields \cite{Davenport1971} and the $2$-part of the class group of cubic fields \cite{Bhargava2005}.
 See \cite[Section 10]{Wood2016} for a detailed introduction to  the Malle-Bhargava principle, which comes from heuristics presented in
\cite[Section 4]{Malle2004} and \cite[Section 8]{Bhargava2010}.

For simplicity, we work here only over $Q=\Q$. 
Given a finite group $G$, and a type $G'$, let $c$ be the set of order $2$ elements of $G'$ that are not in $G\times G$.
Note that $c\cup 1$ is the set of elements of $G'$ that generate a cyclic subgroup intersecting trivially with $G\times G$.
 The numerator $$\sum_{K\in IQ_X}  \#\{\rho\in \Sur(G_K^{\un,\v}, G) | \rho \textrm{ is type }G'\}    $$ in the definition of $E^-(G,G')$ equivalently counts elements  $\psi\in\Sur(\Gal(\bar{\Q}/\Q), G')$ such that for each finite place $v$ the restriction
$\psi_v : \Gal(\bar{\Q}_v/\bar{\Q}_v)\ra G'$ has inertia subgroup a subset of $c\cup\{1\}$, and further the image of $\psi_\infty:\Gal(\C/\R)\ra G'$ is the subgroup generated by $(1,\sigma)$.
 The Malle-Bhargava principle models the count of such $\rho$ by the following Dirichlet series
$$
D(s)=\sum_{n\geq 1} d_n n^{-s}=\left( \frac{1}{|G'|}\sum_{\substack{\psi_\infty : \Gal(\bar{\C}/\bar{\R})\ra G'\\
\im(\phi_\infty)=\langle (1,\sigma) \rangle
}}
(\Disc \psi_\infty)^{-s}\right)
\prod_{v \textrm{ finite}} \left( \frac{1}{|G'|}\sum_{\substack{\psi_v : \Gal(\bar{\Q}_v/\bar{\Q}_v)\ra G'\\
\psi_v(I_v)\cap G\times G=1
}}
(\Disc \psi_v)^{-s}\right)
,
$$
where the product is over all finite places $v$ of $\Q$, and $I_v$ is the inertia subgroup of $\Gal(\bar{\Q}_v/\bar{\Q}_v)$,
and $\Disc \psi_v$ is the norm of the discriminant of the quadratic \'etale algebra associated to the composite of 
$\psi_v$ and $G'\ra S_2$. 
(The factor for $\infty$ on the left is equal to $|G'|^{-1}$, as there is $1$ choice of $\psi_\infty$ such that $\im(\phi_\infty)=\langle (1,\sigma) \rangle$ and $\Disc \psi_\infty=1$ for any $\psi_\infty$.) 
 The idea is that the sums are over over the possible local restrictions $\psi_v$ of the $\psi$ we wish to count, and if we assume some sort of average local-global compatibility, we would conjecture that
$$
\sum_{K\in IQ_X}  \#\{\rho\in \Sur(G_K^{\un,\v}, G) | \rho \textrm{ is type }G'\}    
$$ 
has the same asymptotics as $\sum_{n\leq X}d_n$ up to a constant factor.
Since in this section we are only going to predict asymptotics up to a constant factor, we can ignore all places $v$ with $v\mid |G'|$ (which includes $v=2$) or $v$ infinite.  If we write $G^t_{\Q_v}$ for the Galois group of a maximal tame algebraic extension of $\Q_v$, then $G^t_{\Q_v}$ is isomorphic to the profinite completion of the free group on $x,y$ modulo the relation $xyx^{-1}=y^{\Nm v}$ \cite[Theorem 7.5.2]{Neukirch2000}, and the inertia subgroup is topologically generated by $y$.  
Thus, the $\psi_v : \Gal(\bar{\Q}_v/\bar{\Q}_v)\ra G'$ such that $\psi_v(I_v)\cap G\times G=1$ are given exactly by choices of an order $2$ element $y\in G'\setminus G\times G$, and an $x$ such that $xyx^{-1}=y$ (since $y$ is order $2$ and $v\ne 2$).  For each $y$, the number of choices of $x$ is $|G'|/\#\{zyz^{-1} | z\in G'\}$ by the orbit-stabilizer theorem.   Thus,
$$
\frac{1}{|G'|}\sum_{\substack{\psi_v : \Gal(\bar{\Q}_v/\bar{\Q}_v)\ra G'\\
\psi_v(I_v)\cap G\times G=1
}}
(\Disc \psi_v)^{-s} =1+N_{G'}v^{-s}
$$
where $N_{G'}$ is the number of $G'$ conjugacy classes of order $2$ elements in $G'\setminus (G\times G)$. 
It follows that for some constant $K_{G'}$,
$$
\sum_{n\leq X}d_n =K_{G'} X(\log X)^{N_{G'}-1}+o(X(\log X)^{N_{G'}-1}).
$$
Since $\sum_{K\in IQ_X} 1= {\frac{1}{2\zeta(2)}}X+o(X),$ it follows that the Malle-Bhargava principle predicts the division in Conjecture~\ref{C:E} into finite versus infinite averages.  Moreover, it predicts that when $G'$ is not good that
$$
\frac{\sum_{K\in IQ_X}  \#\{\rho\in \Sur(G_K^{\un,\v}, G) | \rho \textrm{ is type }G'\} }{\sum_{K\in IQ_X} 1}
$$
is asymptotic to $(\log X)^{N_{G'}-1}$, in agreement with the lower bound we obtain in Theorem~\ref{T:FF}. 
We thus ask about the following refinement of Conjecture~\ref{C:E}.
\begin{question}\label{Q:nc}
When $G$ is a finite group and $G'$ and admissible subgroup of $G\wr S_2$,  are there constants $c^{\pm}_{G,G'}>0$ such that 
$$
\lim_{X\ra\infty}
\frac{E^{\pm}_X(G,G')}{c^\pm_{G,G'}(\log X)^{N_{G'}-1} }=1?
$$
\end{question}
Theorem~\ref{T:FF} provides function field evidence for an affirmative answer to the lower bound of Question~\ref{Q:nc}.  In Section~\ref{S:Ex}, we prove an affirmative answer to Question~\ref{Q:nc} in the case $(G,G')=(C_2^k, C_2^{k+1})$.  See also \cite{Alberts2016a,Klys2017} for further evidence that the answer to 
Question~\ref{Q:nc} is yes. The argument for Theorem~\ref{T:FF} could be used to suggest a value for $c^\pm_{G,G'}$ in the function field case, but it would not be clear how to make the analogous constant in the number field case.

\subsection{True and false predictions}
The Malle-Bhargava principle over $\Q$, which in particular implies Malle's Conjecture \cite[Conjecture 1.1]{Malle2004}, makes many predictions that are known to be true, including asymptotics of $S_3$ cubic \cite{Davenport1971}, $S_4$ quartic \cite{Bhargava2005}, $S_5$ quintic \cite{Bhargava2010a}, $S_3$-sextic \cite{Bhargava2008,Belabas2010a}, and elementary abelian \cite{Wright1989, Wood2010} extensions with local conditions at any finite set of places.
The principle gives the correct order of magnitude for the count of $D_4$ quartic fields \cite{Cohen2002} and all abelian extensions \cite{Maki1985, Wright1989} (except for local conditions that are known to never be satisfied).  Moreover, the principle used with infinitely many local conditions makes a provably correct prediction
for cubic fields corresponding, via class field theory, to the $3$-part of the class group of quadratic fields \cite{Davenport1971} and quartic fields corresponding to the $2$-part of the class group of cubic fields \cite{Bhargava2005}.  The application to  the $3$-part of the class group of quadratic fields is a special case of our application.  More generally, the principle also makes correct predictions with infinitely many local conditions for quadratic, $S_3$ cubic, $S_4$ quartic,  $S_5$ quintic fields \cite{Bhargava2014b}, and abelian  \cite{Frei2015} fields with local conditions at infinitely many primes such that for sufficiently large primes all unramified and minimally ramified local extensions are allowed (in the abelian case excepting local conditions that are never satisfied; see also \cite{Bhargava2013a,Taniguchi2013} in the cubic case). 

  However, the Malle-Bhargava principle has counterexamples, including Kl\"uners' \cite{Kluners2005} counterexample to Malle's conjecture for $C_{\ell} \wr S_2$ fields.  
In these examples, $G$ extensions with a particular fixed intermediate extension have the same asymptotic growth rate as the conjecture predicts for all $G$ extensions.  Kl\"uners' \cite{Kluners2005} and T\"urkelli \cite{Turkelli2015} both suggest ways to correct Malle's conjecture, and in particular suggest that it should still hold if we do not count  extensions with nontrivial subfields that are subfields of certain cyclotomic fields.  First, we note that the only \sa subgroup of $C_{\ell} \wr S_2$ is $D_\ell \sub C_{\ell} \wr S_2$ (with $C_\ell \ra C_\ell \times C_\ell$ as $a\mapsto (a,-a)$), so have   not above applied the Malle-Bhargava principle in the cases of known counter-examples.    The following lemma will help us see more generally that the source of these counter-examples is not an issue in our case.
    
  \begin{lemma}\label{L:goodab}
  If $G'$ is a \gd subgroup of $G\wr S_2$, then the abelianization of $G'$ is $S_2$.
  \end{lemma}
  \begin{proof}
  Let $N$ be the set of $g\in G' \cap G\times G$ such that $\sigma(g)=g^{-1}.$  Then $\{(g,\sigma) | g\in N \}$ are all order $2$ elements of $G'$ not in $\ker(G'\ra S_2)$ and so are all conjugate by the definition of a \gd subgroup.  Moreover, any order $2$ element  $G'$ not in $\ker(G'\ra S_2)$ is of the form $(g,\sigma)$ for some $g\in G\times G$ such that $\sigma(g)=g^{-1}.$  Thus by the definition of \sa, the elements $\{(g,\sigma) | g\in N \}$ generate $G'$ and thus $(G')^{ab}$.  Since these elements are all conjugate, it follows $(G')^{ab}$ is generated by $(1,\sigma)$, which has order $2$.  Since by construction, there is a surjection $G'\ra S_2$, the lemma follows.
    \end{proof}
  Thus in $G'$-extensions of $\Q$ for good $G'$, the only intermediate cyclotomic field is the quadratic field $K$ corresponding to $\ker(G'\ra S_2)$.
  Since the $G'$-extensions we count are all unramified over the quadratic field $K$ above, a single $K$ only contributes finitely many extensions to the count, and there are no counter-examples of the type found in \cite{Kluners2005}, and the suggested corrections to Malle's conjecture would not change it in this case.  When $G'$ is not good, any of the suggested corrections to Malle's conjecture would only increase the number of extensions and not affect our Conjecture~\ref{C:E}.
    
Moreover, the Malle-Bhargava principle sometimes makes incorrect predictions for the counts of fields with given local conditions.  
When $G$ is abelian, there are certain local conditions that are never satisfied by global extensions (like Wang's counterexample to Grunwald's ``theorem,'' see \cite{Wright1989, Wood2010}). Also when $G$ is abelian, sometimes local conditions that are sometimes satisfied globally still do not occur with the probabilities that the principle predicts (see \cite[Propositions 1.4 and 1.5]{Wood2010}).   However, for abelian $G$ extensions with finitely many local conditions (or infinitely many local conditions such that for sufficiently large primes all unramified and minimally ramified local extensions are allowed) as long as there are any such extensions at all, the count is still the correct order of magnitude \cite{Wright1989, Frei2015}, and the order of magnitude is all we use here.

\section{Example: $G=C_2^k$ and $G'=C_2^{k+1}$}\label{S:Ex}

In this section, we prove that $E^{\pm}(G,G')=\infty$ when $G=C_2^k$ and $G'=C_2^{k+1}$ for $k\geq 1$.  This $G'$ is realized in $G\wr C_2=(G\times G)\rtimes C_2$ as elements of the form $(a,a,b)$.  Finding $E^{\pm}(G,G')$ amounts to counting $C_2^{k+1}$ extensions of $\Q$ whose inertia has trivial intersection with $C_2^k \times 1 \sub C_2^{k+1}$.  The strategy for counting abelian extensions goes back to Cohn \cite{Cohn1954}, and has been more recently developed in \cite{Wright1989,Wood2010,Frei2015}.

By class field theory, $C_2^{k+1}$ extensions of $\Q$ correspond to continuous homomorphisms from the id\`ele class group $J_\Q$ to $C_2^{k+1}$.  It is an easy exercise to see that the restriction map
\begin{align*}
\Hom_{cts}(J_\Q, C_2^{k+1}) &\ra \Hom_{cts}({\prod_p}' \Z_p^*, C_2^{k+1}) \\
\phi &\mapsto \prod_p \phi|_{\Z_p^*}
\end{align*}
(where the restricted product is over finite primes $p$ of $\Q$)
is a bijection.
Note that $\phi$ corresponds to an extension counted by the numerator of $E(G,G')$ if and only if 
at all finite primes $p$ we have $\phi(\Z_p^*)\cap (C_2^k \times 1 )=1$
and also $\phi(\R^*_{< 0})\cap (C_2^k \times 1 )=1$ .  This condition at a finite prime $p$ happens exactly when
$\phi(\Z_p^*)=1$ or $\phi(\Z_p^*)$ is  order $2$ and generated by an element of $C_2^{k+1}$ that is non-trivial in the last factor.
The condition at the infinite place happens if and only if $\prod_p  \phi|_{\Z_p^*}(-1)=1$ (in which case the extension is totally real) or if $\prod_p  \phi|_{\Z_p^*}(-1)$ is non-trivial in the last factor of $C_2$ (in which case the extension has a complex quadratic subfield corresponding to the quotient to the last factor of $C_2$).

Let $a_n$ be the number of surjections from $\Gal(\bar{\Q}/\Q)$ to $C_2^{k+1}$ such that at all finite primes $p$ we have $\phi(\Z_p^*)\cap (C_2^k \times 1 )=1$ and $\prod_p  \phi|_{\Z_p^*}(-1)=1$, such that the associated quadratic field (from the composition of the surjection with projection onto the last factor of $C_2^{k+1}$) has absolute discriminant $n$.  Then by Proposition~\ref{P:trans} and the above,
$$
\sum_{K\in RQ_{=n}}  \#\{(\rho,y) | \rho\in \Sur(G_K^{\un,\v}, G), y\textrm{ \twist for }\rho , (\rho,y) \textrm{ is type }G'\}=a_n.
$$
For a $\phi\in \Hom_{cts}(\Z_p^*,C_2^{k+1})$, let $\bar{\phi}$ be the composition of $\phi$ with the projection onto the last factor of $C_2$.
From the above we see that
$$
\prod_p \left(\sum_{\substack{\phi\in \Hom(\Z_p^*,C_2^{k+1})\\
\phi(\Z_p^*)\cap (C_2^k \times 1 )=1}} \Disc(\bar{\phi})^{-s}\right)
$$
has coefficients counting homomorphisms from $\Gal(\bar{\Q}/\Q)$ to $C_2^{k+1}$ such that at all finite primes $p$ we have $\phi(\Z_p^*)\cap (C_2^k \times 1 )=1$.
We consider 
$$
F_k(s)=2^{-(k+1)}\sum_{\chi:C_2^{k+1}\ra \pm 1}\prod_p \left(\sum_{\substack{\phi\in \Hom(\Z_p^*,C_2^{k+1})\\
\phi(\Z_p^*)\cap (C_2^k \times 1 )=1}} \chi(\phi(-1))\Disc(\bar{\phi})^{-s} \right),
$$
whose coefficients count homomorphisms from $\Gal(\bar{\Q}/\Q)$ to $C_2^{k+1}$ such that at all finite primes $p$ we have $\phi(\Z_p^*)\cap (C_2^k \times 1 )=1$ and $\prod_p  \phi|_{\Z_p^*}(-1)=1$ 
At all odd primes $p$, if $\chi=1$ we have
$$
\sum_{\substack{\phi\in \Hom(\Z_p^*,C_2^{k+1})\\
\phi(\Z_p^*)\cap (C_2^k \times 1 )=1}} \chi(\phi(-1))\Disc(\bar{\phi})^{-s}=1+2^kp^{-s}.
$$
At all odd primes $p$ such that $-1$ is a square mod $p$, for all $\chi$ we have 
$$
\sum_{\substack{\phi\in \Hom(\Z_p^*,C_2^{k+1})\\
\phi(\Z_p^*)\cap (C_2^k \times 1 )=1}} \chi(\phi(-1))\Disc(\bar{\phi})^{-s}=1+2^kp^{-s}.
$$
At all odd primes $p$ such that $-1$ is not square mod $p$,  if $\ker \chi=C_2^k \times 1$, then
$$
\sum_{\substack{\phi\in \Hom(\Z_p^*,C_2^{k+1})\\
\phi(\Z_p^*)\cap (C_2^k \times 1 )=1}} \chi(\phi(-1))\Disc(\bar{\phi})^{-s}=1-2^kp^{-s}.
$$
and for any other $\chi\neq 1$ we have
$$
\sum_{\substack{\phi\in \Hom(\Z_p^*,C_2^{k+1})\\
\phi(\Z_p^*)\cap (C_2^k \times 1 )=1}} \chi(\phi(-1))\Disc(\bar{\phi})^{-s}=1+2^{k-1}p^{-s}-2^{k-1}p^{-s}=1.
$$
From this analysis, standard methods (see e.g. \cite{Cohn1954,Wood2010}) show that
$F_k(s)$ is analytic for $\Re(s)>1$ and has a pole of order $2^k$ at $s=1$.  Then from standard Tauberian theorems (e.g. \cite[Corollary, p. 121]{Narkiewicz1983}), we see that the coefficients of $F_k(s)$ have partial sums asymptotic to $c_kX (\log X)^{2^k -1}$.  Note that except for the trivial homomorphism, the homomorphisms counted by the coefficients of $F_k(s)$ have non-trivial image in the last factor of $C_2$, because all allowed inertia has such non-trivial image.  However, some of the homomorphisms counted by $F_k(s)$ will have image that is not all of $C_2^k$.  An upper bound for the number of these, however, can be given in terms of a constant (in terms of $k$) and the asymptotics of the coefficients of $F_{k-1}(s).$
It thus follows, using Proposition~\ref{P:trans} that for some positive constant $c'_k$ we have
$$
\sum_{K\in RQ_X}  \#\{(\rho,y) | \rho\in \Sur(G_K^{\un,\v}, G), y\textrm{ \twist for }\rho , (\rho,y) \textrm{ is type }G'\}  =c'_k X (\log X)^{2^k -1}+o(X (\log X)^{2^k -1})
$$
and so 
$\tilde{E}^+(G,G')=\infty$.
Similarly, we can show that 
$$
\sum_{K\in IQ_X}  \#\{\rho\in \Sur(G_K^{\un,\v}, G) | \rho \textrm{ is type }G'\} =c''_k X (\log X)^{2^k -1}+o(X (\log X)^{2^k -1})
$$
and so 
$\tilde{E}^-(G,G')=\infty$.

\section{Charts of values}\label{S:chart}
In this section, we give two charts of the values of invariants important in this paper for small groups.
The computations for these two charts were done in GAP \cite{GAP4}.
  For groups that we give as semi-direct products, we also give their identifier  in the Small Groups Library \cite{SmallGroups} in square brackets.  We write $D_n$ for the dihedral group of order $n$, and $C_n$ for the cyclic group of order $n$, and $Q_n$ for the quaternion group of order $n$, and $A_d$ for the alternating group on $d$ elements.
\newpage
\subsection{Types}\label{S:types}

The below is a chart of all groups $G$ from order $2$ to $15$ as well as the two smallest non-abelian simple groups, and all \sa $G'\sub G \wr S_2$.
For a given $G$, we list each abstract isomorphism class of $G'$ only once.  Even though in principle $G'$ could sit as a subgroup of 
$G \wr S_2$ in two different ways (not in the same $\Aut(G)$ class) such that $|c/G'|$ had different values, that doesn't occur for groups of these orders.  The types $G'$ that are good are those for which the number $|c/G'|$ of conjugacy classes in $c$ is $1$. See \cite[Sections 2 and 3]{Alberts2016} for discussion of which $G$ have an \sa $G'$ with $|G'|=2|G|$.

\begin{center}
  \begin{tabular}{  c | c | c }
    $G$ & $G'$ & $|c/G'|$ \\ \hline\hline
    $C_2$ & $C_2^2$ & 2 \\ \hline 
    $C_3$ & $S_3$ & 1 \\ \hline 
 $C_4$ & $D_8$ & 2\\ \hline 
$C_2^2$ & $C_2^3$ & 4\\ \hline 
$C_5$ & $D_{10}$ & 1\\ \hline 
$S_3$ & $D_{12}$ & 2\\ \hline 
$S_3$ & $S_3 \times S_3$ & 2\\ \hline 
$C_6$ & $D_{12}$ & 2\\ \hline 
$C_7$ & $D_{14}$ & 1\\ \hline 
$C_8$ & $D_{16}$ & 2\\ \hline 
$C_4 \times C_2$ & $C_2 \times D_8$ & 4\\ \hline 
$D_8$ & $C_2 \times D_8$ & 4\\ \hline 
$D_8$ & $(C_2 \times D_8) \rtimes C_2$  [49] & 4\\ \hline 
$Q_8$ & $(C_4 \times C_2) \rtimes C_2$ [13] & 3\\ \hline 
$Q_8$ & $(C_2 \times Q_8) \rtimes C_2$ [50] & 4\\ \hline 
$C_2^3$ & $C_2^4$ & 8\\ \hline 
$C_9$ & $D_{18}$ & 1\\ \hline 
$C_3^2$ & $(C_3^2) \rtimes C_2$ [4] & 1\\ \hline 
$D_{10}$ & $D_{20}$ & 2\\ \hline 
$D_{10}$ & $D_{10}^2 $ & 2\\ \hline 
$C_{10}$ & $D_{20}$ & 2\\ \hline 
$C_{11}$ & $D_{22}$ & 1\\ \hline 
$C_3 \rtimes C_4$ [1] & $(C_6 \times C_2) \rtimes C_2$ [8] & 2\\ \hline 
$C_3 \rtimes C_4$ [1]  & $(C_6 \times S_3) \rtimes C_2$ [22] & 2\\ \hline 
$C_{12}$ & $D_{24}$ & 2\\ \hline 
$A_4$ & $S_4$ & 1\\ \hline 
$A_4$ & $((C_2^4) \rtimes C_3) \rtimes C_2$ [227] & 1\\ \hline 
$D_{12}$ & $C_2^2 \times S_3$ & 4\\ \hline 
$D_{12}$ & $C_2 \times S_3^2 $ & 4\\ \hline 
$C_6 \times C_2$ & $C_2^2 \times S_3$ & 4\\ \hline 
$C_{13}$ & $D_{26}$ & 1\\ \hline 
$D_{14}$ & $D_{28}$ & 2\\ \hline 
$D_{14}$ & $D_{14}^2 $ & 2\\ \hline 
$C_{14}$ & $D_{28}$ & 2\\ \hline 
$C_{15}$ & $D_{30}$ & 1 \\ \hline
$A_5$ & $S_5$ & 1\\ \hline 
$A_5$ & $A_5\times C_2$ & 2\\ \hline 
$A_5$ & $A_5\wr C_2$ & 1\\ \hline 
$\PSL(3,2)$ & $\PSL(3,2)\rtimes C_2$ [208] & 1\\ \hline 
$\PSL(3,2)$ & $\PSL(3,2)\times C_2$  & 2\\ \hline 
$\PSL(3,2)$ & $\PSL(3,2)\wr C_2$  & 1
  \end{tabular}
\end{center}

\subsection{$H_2(G',c)$}\label{S:Hchart}

The below is a chart of all groups $G$ from order $2$ to $31$ as well as the two smallest  non-abelian simple groups, and all \sa $G'\sub G \wr S_2$ such that 
$|c/G'|=1$.  For these pairs, we give $H_2(G',c)$ and the size of the center $|Z(G')|$.
For a given $G$, we list each abstract isomorphism class of $G'$ only once.  Even though in principle $G'$ could sit as a subgroup of 
$G \wr S_2$ in two different ways (not in the same $\Aut(G)$ class) such that $H_2(G',c)$ or $|Z(G')|$ had different values, that doesn't occur for groups of these orders with $|c/G'|=1$. 

\begin{center}
  \begin{tabular}{  c | c | c |c }
    $G$ & $G'$ & $H_2(G',c)$ & $|Z(G')|$ \\ \hline\hline
$C_3$ & $S_3$ & $1$ & 1\\ \hline 
$C_5$ & $D_{10}$ & $1$ & 1\\ \hline 
$C_7$ & $D_{14}$ & $1$ & 1\\ \hline 
$C_9$ & $D_{18}$ & $1$ & 1\\ \hline 
$C_3^2 $ & $(C_3^2) \rtimes C_2$ [4] & $C_3$ & 1\\ \hline 
$C_{11}$ & $D_{22}$ & $1$ & 1\\ \hline 
$A_4$ & $S_4$ & $1$ & 1\\ \hline 
$A_4$ & $((C_2^4) \rtimes C_3) \rtimes C_2$ [227] & $C_2$ & 1\\ \hline 
$C_{13}$ & $D_{26}$ & $1$ & 1\\ \hline 
$C_{15}$ & $D_{30}$ & $1$ & 1\\ \hline 
$C_{17}$ & $D_{34}$ & $1$ & 1\\ \hline 
$C_{19}$ & $D_{38}$ & $1$ & 1\\ \hline 
$C_7 \rtimes C_3$ [1] & $((C_7^2 ) \rtimes C_3) \rtimes C_2$ [7] & $1$ & 1\\ \hline 
$C_{21}$ & $D_{42}$ & $1$ & 1\\ \hline 
$C_{23}$ & $D_{46}$ & $1$ & 1\\ \hline 
$\SL(2,3)$ & $\GL(2,3)$ & $1$ & 2\\ \hline 
$\SL(2,3)$ & $((Q_8^2) \rtimes C_3) \rtimes C_2$ [18130] & $1$ & 2\\ \hline 
$C_{25}$ & $D_{50}$ & $1$ & 1\\ \hline 
$C_5^2 $ & $(C_5^2) \rtimes C_2$ [4] & $C_5$ & 1\\ \hline 
$C_{27}$ & $D_{54}$ & $1$ & 1\\ \hline 
$C_9 \times C_3$ & $(C_9 \times C_3) \rtimes C_2$ [7] & $C_3$ & 1\\ \hline 
$(C_3^2) \rtimes C_3$ [3] & $((C_3^2) \rtimes C_3) \rtimes C_2$ [8] & $1$ & 3\\ \hline 
$(C_3^2) \rtimes C_3$ [3] & $(C_3 \times ((C_3^2) \rtimes C_3)) \rtimes C_2$ [46] & $C_3^2$ & 3\\ \hline 
$C_9 \rtimes C_3$ [4] & $((C_9 \times C_3) \rtimes C_3) \rtimes C_2$ [17] & $1$ & 3\\ \hline 
$C_3^3$ & $(C_3^3) \rtimes C_2$ [14] & $C_3^3$ & 1\\ \hline 
$C_{29}$ & $D_{58}$ & $1$ & 1\\ \hline 
$C_{31}$ & $D_{62}$ & $1$ & 1\\ \hline
$A_5$ & $S_5$ & 1& 1\\ \hline 
$A_5$ & $A_5\wr C_2$ & $C_2$ & 1\\ \hline 
$\PSL(3,2)$ & $\PSL(3,2)\rtimes C_2$ [208] & 1 & 1\\ \hline 
$\PSL(3,2)$ & $\PSL(3,2)\wr C_2$ & $C_2$  & 1
  \end{tabular}
\end{center}

\subsection*{Acknowledgements} 
I would like to thank Brandon Alberts, Manjul Bhargava, Nigel Boston,  Jordan Ellenberg, and Gunter Malle for helpful conversations and comments on the manuscript.  I would like to thank the anonymous referees for numerous comments that improved the quality of the manuscript.  
Thank you to the American Institute of Mathematics for hosting a workshop on the Cohen-Lenstra heuristics in 2011, at which the question was presented that inspired this research.
This work was done with the support of an American Institute of Mathematics Five-Year Fellowship, a Packard Fellowship for Science and Engineering, a Sloan Research Fellowship,  National Science Foundation grants DMS-1301690 and DMS-1652116, and a Vilas Early Career Investigator Award.

\appendix

\section{Empirical averages of the number of unramified $A_4$ extensions of imaginary quadratic fields}

\begin{center}
{Melanie Matchett Wood and Philip Matchett Wood}
\end{center}

\vspace{12pt}

Theorem~\ref{T:FF} gives evidence for the function field analog of Conjecture~\ref{C:E}, and in this appendix we give empirical data that also supports Conjecture~\ref{C:E}.  
Conjecture~\ref{C:E} says that (for good $G'$) we have $\tilde{E}(G,G')=|H_2(G',c)[2]|$. 
 We have described the term $|H_2(G',c)[2]|$ above as a ``correction'' for the $2$ roots of unity in $2$. 
 Recall from the introduction that whenever $G$ is abelian and of odd order, then $|H_2(G',c)[2]|=1$ and Conjecture~\ref{C:E} is implied by the Cohen-Lenstra heuristics for imaginary quadratic fields.  We can see from the chart in Section~\ref{S:Hchart} that
$|H_2(G',c)[2]|=1$ in many other cases.  Since the factor of $|H_2(G',c)[2]|$ is what is most surprising about Conjecture~\ref{C:E}, we thought it most interesting to look at numerical data for a case when $|H_2(G',c)[2]|\ne 1.$

From the chart in Section~\ref{S:Hchart}, we can see that the smallest $G$ for which we have a non-trivial $|H_2(G',c)[2]|$ is when $G=A_4$, and $G'$ is a certain group of order $96$.  We can describe $G'$ explicitly as the subgroup of $A_4 \wr C_2=(A_4\times A_4)\rtimes C_2$ composed of elements whose two $A_4$ coordinates sum to $0$ in the quotient $A_4\ra C_3$.  In this case $\tilde{E}(A_4,G')$ counts the average number of unramified extensions of imaginary quadratic fields that are Galois over the quadratic field with Galois group $A_4$,
but that are not Galois over $\Q$.  This is because we see from the chart in Section~\ref{S:Hchart} that there are only two possible $G'$, and when $G'=S_4$, we can see from the degrees of the groups that the $A_4$-extension of a quadratic must be Galois over $\Q$.  (We also note that the equality $\tilde{E}(A_4,S_4)=1$ is a theorem of Bhargava \cite[Theorem 1.4]{Bhargava2014b}.)

We compute exact or approximate values for this average when we only consider imaginary quadratic fields with discriminant at least $-X$, that is we compute or estimate
     $$
   \tilde{E}^-_{\leq X}(G,G')=\frac{\sum_{K\in IQ_{X}}
   \#\{\rho\in \Sur(G_\Q^{\un,\v}, G) | \rho \textrm{ is type }G'\}    }{\sum_{K\in IQ_{X}} 1}.
    $$
Recall that $\lim_{X\ra\infty} \tilde{E}^-_{\leq X}(G,G') =\tilde{E}(G,G')$.

\subsection{Description of the computation: theoretical}
If $K$ is a quadratic extension of $\Q$, and $L$ is an unramified $A_4$-extension of $K$, then since $A_4$ has a unique normal subgroup with quotient $C_3$, we see that $L$ contains a subfield $M$ such that $M/K$ is a cyclic cubic unramified extension.  It follows from class field theory that $M/\Q$ is an $S_3$ extension.  Each such $M$ contains a triple of isomorphic, non-Galois cubic fields, and we call one of them $M_3$.  Since $M/K$ is unramified, it is straightforward to check that $\Disc M_3=\Disc K$ and that $M_3/\Q$ is not totally ramified at any prime.  Moreover, if $M'_3/\Q$ is a cubic extension not totally ramified at any prime, it is straightforward that check that $M'_3$ is non-Galois and its Galois closure $M'$ contains quadratic subfield $K'$ such that $M'/K'$ is unramified.
(This approach to understanding cyclic cubic unramified extensions of quadratic fields in terms of non-Galois cubic fields can be found in the work of Davenport and Heilbronn \cite{Davenport1971}, along with more details about this correspondence.)   

Our algorithm thus begins with $M_3/\Q$ non-Galois cubic extensions.  We use Belabas's cubic software (described in \cite{Belabas1997} and available on his webpage) to create complete tables of isomorphism classes of cubic fields with discriminant in the interval $[-X,-1]$.  Then we use PARI/GP \cite{PARI2} to determine if the cubic field $M_3$ is nowhere totally ramified.  We determine that by first checking that $\Disc M_3$ is square-free, except for possible powers of $2$.  (A totally ramified prime $p$ will give $p^2\mid \Disc M_3$, and unless $p=2$, if $p^2\mid \Disc M_3$ then $p$ is totally ramified in $M_3$.)  Then we check whether $M_3$ is totally ramified at $2$.   The result is a complete list of nowhere totally ramified cubic fields $M_3$ with $\Disc M_3\in [-X,-1]$.   

We use PARI/GP for the remaining computations.  We compute the Galois closure $M$ of each $M_3$.  This, via the correspondence described above, gives us every unramified cyclic cubic extension $M/K$ of an imaginary quadratic field $K$.  Now if $L/K$ is an unramified $A_4$ extension with subfield $M$, then $L/M$ is an unramified $C_2\times C_2$ extension.  Thus for each $M$ in our list, we find all normal subgroups of the class group  $\operatorname{Cl}(M)$ with $C_2\times C_2$ quotients that are preserved by the action of $\Gal(M/K)$ but not by the action of $\Gal(M/\Q)$.  By class field theory, these correspond exactly to $C_2\times C_2$ unramified extensions $L/M$ such that $L/K$ is Galois but $L/\Q$ is not.  From this it actually follows that $\Gal(L/K)\isom A_4$.  Note that the Schur-Zassenhaus theorem implies that $\Gal(L/K)\isom (C_2\times C_2)\rtimes C_3$ for some action of $C_3$ on $C_2\times C_2$.  There are only two possible actions.  The non-trivial action gives  $\Gal(L/K)\isom A_4$, and the trivial action gives $\Gal(L/K)\isom C_3\times C_3\times C_2$.  However $\Gal(L/K)\isom C_3\times C_3\times C_2$, implies, by class field theory, that $L/\Q$ is Galois.  
So, we conclude that $\Gal(L/K)\isom A_4$.

So the $L$ we have ``found'' (we have actually found the corresponding quotients of $\operatorname{Cl}(M)$) 
are all the unramified extensions $L$ of an imaginary quadratic field $K$ with $\Disc K\in [-X,-1]$ such that $L/K$ is Galois with Galois group $A_4$, but $L/\Q$ is not Galois.  This counts exactly
  $\tilde{E}^-_{\leq X}(A_4,G')$, as described above (for the order 96 group $G'$ described above).

For comparison, we will also compute $\tilde{E}^-_{\leq X}(A,A\rtimes_{-1} C_2)$ exactly for some abelian groups $A$.
These we compute  using the tables of class groups of imaginary quadratic fields created by Mosunov and Jacobson \cite{Mosunov2016} and available on the LMFDB \cite{LMFDB}.  (For these computations we restrict to imaginary quadratic fields of discriminant congruent to $5$ mod $8$, for practical reasons coming from how the tables are given.)

\subsection{Description of the computation: practical}
We ran this computation exhaustively for $X=2^k$ for $1\leq k \leq 23$ 
to compute $\tilde{E}^-_{\leq X}(A_4,G')$ exactly.  However, at this point the exhaustive computation becomes increasingly impractical.  The main bottleneck is PARI/GP computing the class groups of the sextic fields.  

For $X=2^k$ for $21\leq k \leq 32$, we sample fields randomly instead of making an exhaustive search in order to estimate $\tilde{E}^-_{\leq X}(A_4,G')$.  In each interval of the form $(-\ell\cdot 2^{20},-(\ell+1)\cdot 2^{20}]$, we randomly sample $1\%$ of the exhaustive list of cubic fields with discriminant in this range.  (We select the $1\%$ uniformly, with replacement, by generating a random line number from the list of cubic fields.  Since we are sampling from an exhaustive list, the only source of error is the difference in characteristics between what is essentially a perfect random sample of the list and the whole list.)  From each of the chosen cubic fields $M_3$, we continue with the algorithm described above. 

The program described above was scheduled to run in parallel on a pool computers using the HTCondor \cite{Condor1, Condor2, Condor3} high-throughput distributed computing software developed and maintained by the Center for High Throughput Computing at UW-Madison.   The pool of computers on which computations were scheduled consisted of 56 machines running Linux maintained by the UW-Madison Department of Mathematics, which typically allowed for approximately 100 cores running PARI/GP jobs simultaneously (depending on usage load).  The machines in the UW-Madison Math Department pool typically had Intel Core i5 or Xeon CPUs, running at 2.1 GHz to 3.3 GHz, with each machine typically having 4 or 8 cores; note that not all cores are made available for HTCondor scheduling.  
%
%
%
%
%
The total computation time to collect all data 
was just over 3 years, namely 
1115 days, 15 hours, 41 minutes, and 41 seconds.

\subsection{Results of the computation}

We present in Figure~\ref{F} the results of our computation, with the green line giving exactly computed values of $\tilde{E}_{\leq X}(A_4,G')$ and the orange line representing approximations to $\tilde{E}_{\leq X}(A_4,G')$ via sampling.
Note that $X$ is plotted logarithmically.  
The other lines are all exact values for $\tilde{E}_{\leq X}(A,A\rtimes_{-1} C_2),$ (or more precisely its analog for quadratic fields of discriminant congruent to $5$ mod $8$)  where $A$ is an abelian group.  These other three lines are all conjectured, by the Cohen-Lenstra Heuristics \cite{Cohen1984} (see also \cite{Wood2017b} and \cite[Corollary 4]{Bhargava2016}), to go to 1.  The new insight of our paper leads to predicting that, unlike averages of unramified abelian extensions, $\tilde{E}_{\leq X}(A_4,G')$ should instead go to $2$. 

\begin{figure}
\includegraphics[scale=.7]{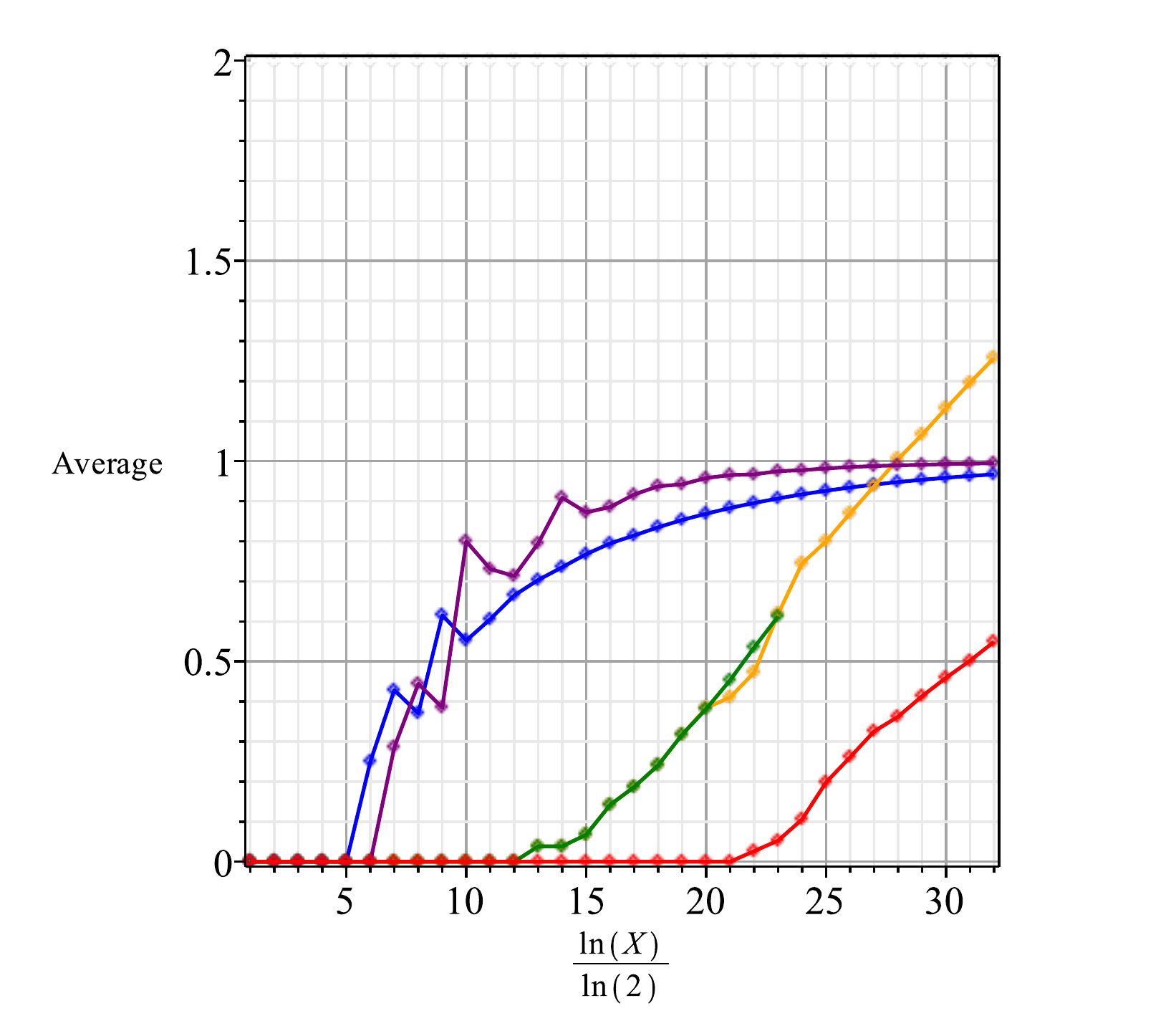}
\caption{Blue: $\tilde{E}_{\leq X}(C_3,S_3)$;  Purple: $\tilde{E}_{\leq X}(C_5,D_{10})$;
Red: $\tilde{E}_{\leq X}(C_3^3,C_3^3\rtimes_{-1} C_2)$; Green, Orange: $\tilde{E}_{\leq X}(A_4,G')$}\label{F}
\end{figure}

Empirically, in these cases and many other similar ones, one observes that these functions generally approach their (conjectured) limits from below, and appear roughly increasing and concave, especially after we pass out of a the range of small numbers. 
We  note that it appears that for small $G$, e.g. $G=C_3,C_5$, the apparent convergence is faster than for $G=A_4$ or $C_3^3$.
 With data up to $X={2^{23}}$, it is very hard to predict from the data whether $\tilde{E}_{\leq X}(A_4,G')$ will go above $1$.  However, with our sampled data, we find very convincing evidence that $\tilde{E}(A_4,G')>1.$
(Our estimation for $\tilde{E}_{\leq 2^{32}}(A_4,G')$ is $1.26$.)  
  This is the main phenomenon that is new in our conjecture, and from the number field side is a priori surprising and perhaps unintuitive (though somewhat explained now by the lifting invariant described above and its connection to the function field analog).

\subsection*{Acknowledgements}
We would like to thank to John Heim for assistance working with the UW-Madison Math Department computer pool, and we would also like to Steve Goldstein for providing many useful examples and explanations on how to use the HTCondor system. P. M. Wood's work on this project was partially supported by NSA grant number H98230-16-1-0301.

\newcommand{\etalchar}[1]{$^{#1}$}
\def\cprime{$'$}

\end{document}